\documentclass[reqno,11pt]{amsart}

\usepackage[utf8]{inputenc}
\usepackage[T1]{fontenc}
\usepackage[english]{babel}

\usepackage{amsmath, latexsym, amsfonts, amssymb, amsthm, amscd}
\usepackage{mathrsfs}
\usepackage{graphics,epsf,psfrag}
\usepackage[lofdepth,lotdepth]{subfig}
\usepackage{graphicx}

\numberwithin{equation}{section}
\newcommand{\dd}{\text{\rm d}}
\DeclareMathOperator{\Supp}{Supp}

\newtheoremstyle{theoreme}
  {10pt}
  {10pt}
  {\itshape}
  {}
  {\bfseries}
  {: }
  { }
  {\thmname{#1}\thmnumber{ #2}\thmnote{ #3}}
  
\newtheoremstyle{lemme}
  {10pt}
  {10pt}
  {\itshape}
  {}
  {\bfseries}
  {: }
  { }
  {\thmname{#1}\thmnumber{ #2}\thmnote{ #3}}
  
\newtheoremstyle{remarque}
  {\topsep}
  {\topsep}
  {\normalfont}
  {}
  {\itshape}
  {: }
  { }
  {\thmname{#1}\thmnumber{ #2}\thmnote{ #3}}

\newtheoremstyle{reponse}
  {\topsep}
  {\topsep}
  {\normalfont}
  {}
  {\itshape}
  {: }
  { }
  {\thmname{#1}\thmnote{ #2}}
  
\newtheoremstyle{theoremeetdef}
  {10pt}
  {10pt}
  {\itshape}
  {}
  {\bfseries}
  {: }
  { }
  {\thmname{#1}\thmnumber{ #2}\thmnote{ #3}}

\theoremstyle{theoreme}
\newtheorem{theorem}{Theorem}[section]

\theoremstyle{lemme}
\newtheorem{lemma}[theorem]{Lemma}
\newtheorem{proposition}[theorem]{Proposition}
\newtheorem{definition}[theorem]{Definition}

\theoremstyle{remarque}
\newtheorem{remark}[theorem]{Remark}

\begin{document}
\title[Quenched large deviations for interacting diffusions]{Quenched large deviations for interacting diffusions in random media}

\author{Eric Lu\c{c}on}
\address{Universit\'e Paris Descartes, Sorbonne Paris Cit\'e, Laboratoire MAP5, UMR 8145, 75270 Paris, France}

\date{\today}
\begin{abstract}
The aim of the paper is to establish a large deviation principle (LDP) for the empirical measure of mean-field interacting diffusions in a random environment. The point is to derive such a result once the environment has been frozen (quenched model). The main theorem states that a LDP holds for every realization of the environment, with a rate function that does not depend on the disorder and is different from the rate function in the averaged model. Similar results concerning the empirical flow and local empirical measures are provided.
  \\[10pt]
  2010 \textit{Mathematics Subject Classification: 60F10, 60K35, 82B44}\\
  \textit{Keywords: mean-field particle systems, large deviation principle, disordered systems, Sanov theorem.}
\end{abstract}

\maketitle


\section{The model and main results}
\subsection{Disordered mean-field interacting diffusions}
Consider, for $T>0$, $N\geq1$, the system of $N$ coupled diffusions: for $t\in[0, T]$,
\begin{equation}
\label{eq:edsi}
x_{i, t} = x_{ i, 0}  + \int_{0}^{t} g(x_{ i, s}, \omega_{ i}) \dd s- \int_{0}^{t} \partial_{ x_{ i}}H_{N}(x_{1, s}, \dots, x_{N, s}, \omega_{ 1}, \ldots \omega_{ N})\dd s  + B_{i, t},\ i=1, \ldots, N,
\end{equation}
where $\{B_{ i}\}_{ i=1, \ldots, N}$ is a sequence of independent standard Brownian motions and where the Hamiltonian $H_{ N}$ is given by :
\begin{equation}
\label{eq:HN}
H_{ N}(x_{ 1}, \ldots, x_{ N}, \omega_{ 1}, \ldots, \omega_{ N}) := \frac{1}{2N} \sum_{k,l=1}^{N}f(x_{k}-x_{l}, \omega_{ k}, \omega_{ l}),
\end{equation}
for some regular functions $f$ and $g$. Here, each trajectory $ \left\lbrace x_{ i, t}\right\rbrace_{ t\in[0, T]}$ is an element of $ \mathcal{ E}:= \mathcal{ C}([0, T], \mathbb{ R})$, the set of continuous functions on $[0, T]$. 

In \eqref{eq:edsi}, both the intrinsic dynamics $g$ and the interaction kernel $f$ are perturbed by a fixed (deterministic) sequence $\{ \omega_{ i}\}_{ i\geq1}$. One has to think of $ \omega_{ i}\in \mathcal{ F}:= \mathbb{ R}$ has a frozen typical realization of an intrinsic disorder associated to the particle $x_{ i}$: \eqref{eq:edsi} is a system of mean-field diffusions in a quenched random environment. Typical examples of \eqref{eq:edsi} can be found in numerous situations in statistical physics (e.g. synchronization \cite{Acebron2005}, neuroscience \cite{Faugeras:2014ab}, social interactions \cite{Collet2010}, etc.).

\subsubsection{Quenched large deviations}

The main results of the paper (Theorem~\ref{theo:LN_x_om} and Theorem~\ref{theo:tilde_LN_x_om} below) concern quenched large deviation principles (LDP) as $N\to \infty$ for the empirical measure 
\begin{equation}
\label{eq:emp_meas_intro}
L_{ N}^{ \underline{ \omega}}= \frac{ 1}{ N} \sum_{ i=1}^{ N} \delta_{ (x_{ i}, \omega_{ i})}
\end{equation}
and the empirical flow $ \mathscr{ L}_{ N}^{ \underline{ \omega}}$ associated with \eqref{eq:edsi}, as well as byproducts of these objects (see below for precise definitions). The superscript $ \underline{ \omega}$ in \eqref{eq:emp_meas_intro} is here to emphasize on the fact that we consider the empirical measure conditioned on the particular value of the disorder sequence $ \underline{ \omega}:=( \omega_{ 1}, \ldots, \omega_{ N})$. 

In \cite{daiPra96}, Dai Pra and den Hollander address, among other results, the issue of large deviations for inhomogeneous mean-field systems similar to \eqref{eq:edsi}, under the joint law of the stochastic noise and the disorder (\emph{averaged large deviations}). The purpose of the present work is to establish a similar result where the disorder sequence has been frozen once and for all (\emph{quenched model}). Namely, we prove in Theorem~\ref{theo:LN_x_om} below that for every deterministic sequence $\{ \omega_{ i}\}_{ i\geq1}$ satisfying appropriate conditions (see Section~\ref{sec:model_assumptions}), the sequence $ \left\lbrace L_{ N}^{ \underline{ \omega}}\right\rbrace_{ N\geq1}$ satisfies a large deviation principle governed by a rate function that does not depend on the sequence $\{ \omega_{ i}\}_{ i\geq1}$ and is different from the averaged rate function found in \cite{daiPra96}. 

The approach of \cite{daiPra96} is simple: when the interaction term $H_{ N}$ is removed in \eqref{eq:edsi}, \emph{under the joint law of the noise and disorder}, the couples $( x_{ i}, \omega_{ i})$ are independent and identically distributed. Hence, the corresponding large deviation principle is provided by Sanov Theorem. The LDP for the whole system with interaction follows from a Girsanov transform and Varadhan's lemma. We follow here a similar strategy, with the notable exception that the initial Sanov argument no longer holds in our quenched framework since, \emph{for a fixed environment $\{ \omega_{ i}\}_{ i\geq1}$}, the particles $\{x_{ i}\}_{ i\geq1}$ are no longer identically distributed. However, based on arguments about projective limits from Dawson and G\"artner \cite{MR885876}, it is still possible to derive a quenched version of a Sanov theorem for the system \eqref{eq:edsi} without interaction. We follow here the approach of Cattiaux and L\'eonard \cite{MR1307957} (see also \cite{2007arXiv0710.1461L}). The structure of rest of the proof is similar to \cite{daiPra96}. An important issue at this point (and a generalization of \cite{daiPra96}) is to cope with possibly unbounded coefficients, as it is the case in various examples in applications.
\subsubsection{ Empirical measures}
For the rest of the paper, for any Polish space $ \mathcal{ X}$, we denote by $ \mathcal{ M}(\mathcal{ X})$ the set of probability measures on $ \mathcal{ X}$. Unless specified otherwise, $ \mathcal{ M}(\mathcal{ X})$ will be endowed with the topology of weak convergence and with its Borel $\sigma$-field.

Under regularity assumptions on the coefficients $f$ and $g$ that we will be made below in Section~\ref{sec:model_assumptions}, the system \eqref{eq:edsi}, endowed with suitable initial conditions, has a unique (strong) solution in $ \mathcal{ E}^{ N}$. We will denote by $P_{N}^{\underline{ \omega}}$ its law, element of $ \mathcal{ M}( \mathcal{ E}^{ N})$. Define also $ \mathcal{ Y}:= \mathcal{ M}(\mathcal{ E}\times \mathcal{ F})$. In the following definitions, $ \underline{u}=(u_{ 1}, \ldots, u_{ N})$ stands for any generic disorder sequence.
\begin{definition}
The empirical measure on both particles and disorder (also addressed as \emph{double-layer empirical measure} in \cite{daiPra96}) is given by
\begin{align}
L^{ \underline{ u}}_{ N}:\mathcal{ E}^{ N} & \to \mathcal{ Y} \nonumber\\
\underline{ x} &\mapsto L_{ N}^{ \underline{u}}[ \underline{ x}]:= \frac{ 1}{ N} \sum_{ i=1}^{ N} \delta_{ (x_{ i}, u_{ i})}, \label{eq:LN_proc}
\end{align}
Denote by 
\begin{align}
\mathbf{ P}_{ N}^{ \underline{u}}&:= P_{ N}^{ \underline{u}}\circ \left(L_{ N}^{ \underline{u}}\right)^{ -1},\label{eq:PuN_LN}
\end{align} 
the law of the empirical measure of the coupled process \eqref{eq:edsi} and disorder.
\end{definition}
Each element of $\mathcal{ M}( \mathcal{ E} \times \mathcal{ F})$ induces an empirical flow through the continuous projection
\begin{align}
\label{eq:mapping_mcE_to_C}
\pi: \mathcal{ M}(\mathcal{ E} \times \mathcal{ F}) & \to \mathcal{ C}([0, T], \mathcal{ M}( \mathbb{ R}\times \mathbb{ R})) \nonumber\\
\lambda &\mapsto \pi \lambda:= \left\{ \lambda_{ t}\right\}_{ t\in [0, T]},
\end{align}
where
\begin{equation}
\label{eq:pit}
\lambda_{ t}(A \times B) := \lambda\left( \left\lbrace (x, \omega),\ x_{ t}\in A, \omega\in B\right\rbrace\right).
\end{equation}
\begin{definition}
The empirical flow associated to \eqref{eq:edsi} is defined by
\begin{align}
\mathscr{ L}_{ N}^{ \underline{ u}}: \mathcal{ E}^{ N} & \to \mathcal{ C}([0, T], \mathcal{ M}(\mathbb{ R} \times \mathbb{ R})), \nonumber\\
 \underline{ x} & \mapsto \left(t \mapsto \frac{ 1}{ N} \sum_{ i=1}^{ N} \delta_{ (x_{ i, t}, u_{ i})}\right). \label{eq:LN_flow}
\end{align}
Denote by
\begin{equation}
\mathbf{ p}_{ N}^{ \underline{u}}:= P_{ N}^{ \underline{u}}\circ \left( \mathscr{ L}_{ N}^{ \underline{u}}\right)^{ -1},\label{eq:PuN_LN_flow}
\end{equation}
the law of the empirical flow of the coupled process \eqref{eq:edsi} and disorder.
\end{definition}
Theorem~\ref{theo:LN_x_om} and Theorem~\ref{theo:tilde_LN_x_om} below respectively address the quenched behavior of $\left\lbrace \mathbf{ P}_{ N}^{ \underline{ \omega}}\right\rbrace_{ N\geq1}$ and $ \left\lbrace \mathbf{ p}_{ N}^{ \underline{ \omega}}\right\rbrace_{ N\geq 1}$ as $N\to \infty$.
\subsection{Notation and assumptions on the model}
\label{sec:model_assumptions}
\subsubsection{Preliminaries on large deviation theory}
\label{sec:large_deviation_theory}
We recall here some usual definitions on large deviation theory. We refer to classical references on the subject, e.g. \cite{Cerf2007} or \cite{Dembo1998}.

Let $\mathcal{X}$ be a regular topological space (\cite[\S~4.1, p.~116]{Dembo1998}) endowed with a regular $\sigma$-field $\mathcal{B}$. A rate function $ \mathcal{ I}$ is a lower semi-continuous mapping $ \mathcal{ I}:\mathcal{X} \to[0, \infty]$ such that each level set
$\Psi_{ \mathcal{ I}}(\alpha):= \left\lbrace x\in\mathcal{X},\ \mathcal{ I}(x)\leq \alpha\right\rbrace$ is closed for all $\alpha\in[0, \infty)$. We say that the rate function $ \mathcal{ I}$ is good when all the level sets $\Psi_{ \mathcal{ I}}(\alpha)$ are compact.

Let $ \left\lbrace \rho_{N}\right\rbrace_{N\geq 1}$ be a sequence of probability measures on $(\mathcal{X}, \mathcal{B})$. The sequence $\left\lbrace \rho_{N}\right\rbrace_{N\geq 1}$ satisfies a strong large deviation principle in $\mathcal{X}$, at speed $N$, with rate function $ \mathcal{ I}$ if for all $A\in\mathcal{B}$,
\begin{equation*}
-\inf_{x\in \mathring{A}} \mathcal{ I}(x) \leq \liminf_{N\to\infty}\frac1N \ln\rho_N(A) \leq \limsup_{N\to\infty}\frac1N \ln\rho_N(A)\leq -\inf_{x\in
\bar{A}} \mathcal{ I}(x),
\end{equation*}
where $ \mathring{A}$ and $ \bar A$ stand for the interior and closure of $A$, respectively. We say that the large deviation principle is weak when the upper bound holds for compact sets only.

Let $ \mathcal{ E}$ a Polish space and endow $\mathcal{X}=\mathcal{M}( \mathcal{ E})$ with the  weak topology. The relative entropy of two elements $\nu$, $\tilde\nu\in \mathcal{ X}$ is given by
\begin{equation}
\label{eq:relative_entropy}
\mathcal{H}(\nu \vert \tilde\nu) := \begin{cases}
\int_{ \mathcal{ E}} \frac{\dd\nu}{\dd\tilde\nu}\ln\left(\frac{\dd\nu}{\dd\tilde\nu}\right)\dd\tilde\nu & \text{ if } \nu\ll\tilde\nu,\\
+\infty&\text{otherwise},
\end{cases}
\end{equation}
where in the case where $\nu$ is absolutely continuous w.r.t. $\tilde\nu$ ($\nu\ll\tilde\nu$), $\frac{\dd\nu}{\dd\tilde\nu}$ stands for the
Radon-Nykodym derivative of $\nu$ w.r.t. $\tilde\nu$.

If $ \mathcal{ E}$ and $ \mathcal{ F}$ are Polish spaces and $ \lambda( \dd x, \dd u)$ a probability measure on $ \mathcal{ E} \times \mathcal{ F}$, we will denote by
\begin{equation}
\lambda_{ 2}(B):= \lambda \left\lbrace (x, u),\ u\in B\right\rbrace
\end{equation}
the marginal of $ \lambda$ on the second coordinate and by
\begin{equation}
\lambda(\dd x, \dd u) = \lambda^{ u}(\dd x) \lambda_{ 2}(\dd u)
\end{equation}
the corresponding regular disintegration.
\subsubsection{ Assumptions on the model}
We suppose that the interaction kernel $ (x, \omega, \tilde{ \omega}) \mapsto f(x, \omega, \tilde{ \omega})$ is $ \mathcal{ C}^{ 2}$ w.r.t. the state variable $x$ such that $f$, $ \partial_{ x}f$ and $\partial_{ x}^{ 2}f$ are bounded in $(x, \omega, \tilde{ \omega})$, with uniform bound $C_{ f}$. We assume also the following symmetry assumption
\begin{equation}
\label{eq:symmetry_f}
f(x, \omega, \tilde{ \omega})= f(-x, \tilde{ \omega}, \omega), x\in \mathbb{ R}, \omega, \tilde{ \omega}\in \mathbb{ R}.
\end{equation}

Concerning the local dynamics $(x, \omega) \mapsto g(x, \omega)$, we suppose the existence of a constant $C_{ g}>0$ and exponents $k_{ 1}, k_{ 2}\geq0$ such that
\begin{align}
\sup_{ x\in \mathbb{ R}} \left\vert g(x, \omega) \right\vert & \leq C_{ g} \left(1+ \left\vert \omega \right\vert^{ k_{ 1}}\right),\label{eq:bound_g_x}\\
\sup_{ x\in \mathbb{ R}} \left\vert g(x, \omega) - g(x, \tilde{ \omega}) \right\vert &\leq C_{ g} \left\vert \omega - \tilde{ \omega}\right\vert\left(1+ \left\vert \omega \right\vert^{ k_{ 2}} + \left\vert \tilde{ \omega} \right\vert^{ k_{ 2}}\right),\ \omega, \tilde{ \omega}\in \mathbb{ R}, \label{eq:Lip_g_om}\\
\sup_{ \omega\in \mathbb{ R}}\left\vert g(x, \omega) -g (y, \omega) \right\vert & \leq C_{ g} \left\vert x-y \right\vert ,\ x,y\in \mathbb{ R},\label{eq:Lip_g_x}
\end{align}
\begin{remark}
The above assumptions cover various classical inhomogeneous models from statistical physics, the main example and motivation of this work being the Kuramoto model for synchronization and its numerous variants:
\begin{itemize}
\item $ g(x, \omega)= \omega$ and $ f(x, \omega, \tilde{ \omega})= - K\cos(x)$, $K\geq0$ corresponds to the standard stochastic Kuramoto model \cite{Acebron2005,Kuramoto1984}. Here, each $x_{ i}$ is considered as a phase, that is an element of the circle $ \mathbb{ S}:= \mathbb{ R}/2 \pi$. Note that, by a direct application of the contraction principle to the projection $ \mathbb{ R} \to \mathbb{ S}$, all the large deviation estimates presented here remain true when the state space $ \mathbb{ R}$ is replaced by $ \mathbb{ S}$.
\item $ g(x, \omega) = \omega$ and $f(x, \omega, \tilde{ \omega})= b(\omega) b( \tilde{ \omega}) \cos(x)$ (for some bounded function $b$) is Daido's variants of the Kuramoto model \cite{daido1987population,Daido:2015aa}, where inhomogeneities are present in the connections between particles,
\item $ g(x, \omega) = 1- a(\omega) \sin(x)$ and $ f(x, \omega, \tilde{ \omega})= - K\cos(x)$ is the \emph{active rotator model} \cite{Sakaguchi1986,doi:10.1137/110846452,MR3336876}. 
\end{itemize}
\end{remark}
Let $ \{\gamma^{ \omega}\}_{ \omega\in \mathbb{ R}}$ be a collection of probability measures on $ \mathbb{ R}$. We suppose that the initial values in \eqref{eq:edsi} are independent random variables satisfying
\begin{equation}
\label{eq:initial_condition_edsN}
x_{ i, 0} \sim \gamma^{ \omega_{ i}},\ i=1,\ldots, N,
\end{equation} 
where the $ \underline{ \omega}:=(\omega_{ 1}, \ldots, \omega_{ N})$ are the $N$ first terms of the disorder sequence. We suppose that $ \omega \mapsto \gamma^{ \omega}$ is Feller, in the sense that for all continuous test function $ \phi$,
\begin{equation}
\omega \mapsto \int_{ \mathbb{ R}} \phi(x) \gamma^{ \omega}(\dd x),
\end{equation}
is continuous.
We also suppose that for all $ \omega$, $ \gamma^{ \omega}$ has exponential moments of every order and introduce:
\begin{equation}
\label{eq:ell_exp_moments}
\ell_{ \omega}:= r \in \mathbb{ R}\mapsto \ln \int e^{ r \left\vert x \right\vert} \gamma^{ \omega}(\dd x).
\end{equation}
We assume the following control of $ \ell_{ \omega}$: there exists $ \tau\geq0$ such that for all fixed $r\in \mathbb{ R}$, for some constant $C(r)$ possibly depending on $r$
\begin{equation}
\label{eq:control_ell}
\ell_{ \omega}(r) \leq C(r) \left(1+ \left\vert \omega \right\vert^{ \tau} \right),\ \omega\in \mathbb{ R}.
\end{equation}
For the rest of the paper, we fix once and for all a disorder sequence $ \left\lbrace \omega_{ i}\right\rbrace_{ i\geq1}$ which verifies the following assumptions: there exists a probability measure $ \mu$ on $ \mathbb{ R}$ such that
\begin{equation}
\label{eq:convergence_mes_emp_mu}
\frac{ 1}{ N} \sum_{ i=1}^{ N} \delta_{ \omega_{ i}} \to \mu,\ \text{ as } N\to\infty,
\end{equation} where the above convergence holds in $ \mathcal{ M}( \mathbb{ R})$ endowed with the topology of weak convergence. We also define $ \gamma\in \mathcal{ M}( \mathbb{ R})$ by
\begin{equation}
\label{eq:gamma}
\int_{ \mathbb{ R}} \phi(x) \gamma(\dd x) = \int_{ \mathbb{ R}} \int_{ \mathbb{ R}} \phi(x) \gamma^{ \omega}(\dd x) \mu(\dd \omega), 
\end{equation}
for all test functions $ \phi$.

Assumption \eqref{eq:convergence_mes_emp_mu} ensures the convergence of $ \frac{ 1}{ N}\sum_{ i=1}^{ N} \phi( \omega_{ i})$ as $N\to\infty$ for any bounded continuous test function $ \phi$. For technical reasons, we need to extend this convergence to some additional test functions. We suppose first that it holds for \eqref{eq:ell_exp_moments}: there exists $p>1$ such that for all $r\in \mathbb{ R}$, the following convergence holds
\begin{equation}
\label{eq:conv_ell}
\frac{ 1}{ N}\sum_{ i=1}^{ N} \ell_{ \omega_{ i}}(r)^{ p} \to_{ N\to\infty} \int_{ \mathbb{ R}} \ell_{ \omega}(r)^{ p} \mu(\dd \omega)<+\infty.
\end{equation}
\begin{remark}
A closer look at the proof below (see Proposition~\ref{prop:limgam}) shows that \eqref{eq:conv_ell} is actually only required to hold for a sequence $\{r_{ n}\}_{ n=1, 2, \ldots}$ such that $r_{ n}\to +\infty$ as $n\to\infty$.
\end{remark}
Second, we suppose the convergence of some empirical moments: there exists $ \iota$ such that $ \iota> k_{ 1}$, $ \iota> \tau$ and $ \iota\geq k_{ 2}+1$ satisfying
\begin{equation}
\frac{ 1}{ N} \sum_{ i=1}^{ N} \left\vert \omega_{ i} \right\vert^{ \iota}\to_{ N\to\infty}\int_{ \mathbb{ R}} \left\vert \omega \right\vert^{ \iota} \mu(\dd \omega)< +\infty, \label{eq:moment_cond_mu}
\end{equation}
\begin{remark}
Assumptions \eqref{eq:control_ell} and \eqref{eq:conv_ell} are not restrictive. Consider for example the case where for all $ \omega$, a variable $ \xi$ with law $\gamma^{ \omega}$ is such that $ \xi= \alpha(\omega) + \zeta$, where $ \alpha(\omega)$ is a deterministic mean value and $ \zeta$ is a (centered) sub-gaussian variable: $ \mathbb{ E}( e^{ r \zeta}) \leq e^{ \beta(\omega)^{ 2} r^{ 2}/2}$, for some $ \beta(\omega)>0$. A rough bound gives $ \mathbb{ E}( e^{ \left\vert r \zeta \right\vert})= \mathbb{ E}(e^{ \left\vert r \right\vert \zeta} \mathbf{ 1}_{ \zeta\geq0}) + \mathbb{ E}( e^{ - \left\vert r \right\vert \zeta} \mathbf{ 1}_{ \zeta<0}) \leq 2 e^{ \beta(\omega)^{ 2} r^{ 2}/2}$. This gives that $ \ell_{ \omega}(r) \leq \left\vert r \right\vert \left\vert \alpha(\omega) \right\vert + \frac{ \beta(\omega)^{ 2} r^{ 2}}{ 2} + const.$ Hence, we easily see that \eqref{eq:control_ell} and \eqref{eq:conv_ell} hold under appropriate polynomial control on $ \alpha$ and $ \beta$ and moment conditions for the measure $ \mu$.
\end{remark}
\subsection{ Main results}
The main result of the paper is a large deviation estimate concerning the empirical measure $L_{ N}^{ \underline{ \omega}}$ defined in \eqref{eq:LN_proc}. In order to state the theorem, we need some further notation: for $q\in \mathcal{ M}( \mathbb{ R}\times\mathbb{R})$ and $\omega\in\mathbb{R}$, denote by $\beta^{q, \omega}$ the function defined on $ \mathbb{ R}$ by:
\begin{equation}
\label{eq:beta_q}
\beta^{q, \omega}(x) = -\int \partial_{ x}f(x- \tilde{ x}, \omega, \tilde \omega) q(\dd \tilde x, \dd\tilde\omega),\ x\in \mathbb{ R}.
\end{equation}
For fixed $\lambda\in \mathcal{ Y}$, consider $P^{\lambda, \omega}\in \mathcal{ M}( \mathcal{ E})$ the law of the unique
(strong) solution to the following SDE:
\begin{equation}
\label{eq:simp1}
\dd x_{t} = \beta^{\lambda_t, \omega} (x_{t}) \dd t + g(x_t, \omega)\dd t + \dd b_{t}; \ x_{0}\sim\gamma^{ \omega},
\end{equation}
where, $b$ is a standard Brownian motion and $ \lambda_{ t}$ is defined in \eqref{eq:pit}.
\begin{definition}
For all measure $ \nu\in \mathcal{ M}( \mathbb{ R})$ such that $ \int_{ \mathbb{ R}} \left\vert \omega \right\vert^{ k_{ 1}} \nu(\dd \omega) <+\infty$, define $P^{\lambda}_{ \nu}\in \mathcal{ Y}$ by: 
\begin{equation}
\label{eq:def_PQ}
 P^{\lambda}_{ \nu}(\dd x, \dd\omega):= P^{\lambda, \omega}( \dd x)\nu(\dd\omega).
\end{equation}
\end{definition}
The main result of the paper is Theorem~\ref{theo:LN_x_om}: 
\begin{theorem}
\label{theo:LN_x_om}
Under the assumptions of Section~\ref{sec:model_assumptions}, for any fixed sequence of disorder $ \left\lbrace \omega_{ i}\right\rbrace_{ i\geq1}$ satisfying \eqref{eq:convergence_mes_emp_mu}, \eqref{eq:conv_ell} and \eqref{eq:moment_cond_mu},  the sequence $ \left\lbrace \mathbf{ P}_{ N}^{ \underline{ \omega}}\right\rbrace_{ N\geq1}$ satisfies a strong large deviation principle in $ \mathcal{ Y}= \mathcal{ M}( \mathcal{ E}\times \mathcal{ F})$ with speed $N$, governed by the good rate function
\begin{equation}
\label{eq:rate_function_proc}
\mathcal{ G}= \mathcal{ G}^{ que}:= \lambda \mapsto \begin{cases}
\mathcal{H}(\lambda\vert  P^{\lambda}_{ \mu})= \int_{ \mathbb{ R}} \mathcal{ H}( \lambda^{ \omega}\vert P^{ \lambda, \omega}) \mu(\dd \omega)& \text{ if } \lambda_{ 2}= \mu,\\
+\infty & \text{ otherwise}.
\end{cases}
\end{equation}
\end{theorem}
\begin{proposition}
\label{prop:mckean_vlasov}
The rate function $ \mathcal{ G}(\cdot)$ has a unique zero $ \lambda^{ \ast}$ which is characterized by $ \lambda^{ \ast}(\dd x, \dd \omega)= \lambda^{ \ast, \omega}(\dd x) \mu(\dd \omega)$ such that for $ \mu$-a.e. $\omega$, $t \mapsto \lambda_{ t}^{ \ast, \omega}(\dd x)$ is the unique weak solution of the McKean-Vlasov equation
\begin{equation}
\label{eq:mckean_vlasov}
\begin{cases}
\lambda_{ 0}^{ \omega}(\dd x)&= \gamma^{ \omega}(\dd x),\\
\partial_{ t} \lambda_{ t}^{ \omega} &= \mathscr{ L}^{ \omega} \lambda_{ t}^{ \omega},
\end{cases}
\end{equation}
where $ \mathscr{ L}^{ \omega}$ is the nonlinear operator
\begin{equation}
\label{eq:mscL_om}
\mathscr{ L}^{ \omega} \lambda^{ \omega}_{ t} := \frac{ 1}{ 2} \partial_{ x}^{ 2}\lambda_{ t}^{ \omega} - \partial_{ x} \left(\beta^{ \lambda_{ t}, \omega} \lambda^{ \omega}_{ t}\right)
\end{equation}
\end{proposition}
\begin{proof}[Proof of Proposition~\ref{prop:mckean_vlasov}]
Existence follows from the fact that a good rate function has always at least one zero. It is easy to see that any zeros $ \lambda$ of the rate function $ \mathcal{ G}(\cdot)$ is a solution to \eqref{eq:mckean_vlasov}, by definition of $P^{ \lambda, \omega}$ (recall \eqref{eq:beta_q} and \eqref{eq:simp1}). Uniqueness of any solution to the McKean-Valsov equation \eqref{eq:mckean_vlasov} is a standard issue: we refer for example to \cite{Oelsch1984}, Lemma~10, \cite{daiPra96}, Th.~2 or \cite{Lucon2011}, Th.~2.5.
\end{proof}
\begin{remark}
\label{rem:McKV_nonlinear}
The McKean-Vlasov equation \eqref{eq:mckean_vlasov} is here understood in a weak sense. Using the parabolic nature of the problem, one can actually prove that the measure-valued solution $ t \mapsto \lambda_{ t}$ of \eqref{eq:mckean_vlasov} has a density $q_{ t}(x, \omega)$ w.r.t. $ \dd x \mu(\dd \omega)$ for any positive time and that this density is a strong solution to \eqref{eq:mckean_vlasov}. See \cite{MR3207725}, A.1 for a precise statement.

A usual point of view is to see the unique solution $ \lambda^{ \ast}$ to \eqref{eq:mckean_vlasov} as the law of $(\bar x^{ \omega}, \omega)$, where $\bar x^{ \omega}$ is the so-called \emph{nonlinear process} defined by \eqref{eq:simp1} in the case where $ \lambda= \lambda^{ \ast}$. In this set-up, an alternative proof of existence of \eqref{eq:mckean_vlasov} relies on a fixed-point argument \cite{SznitSflour}.
\end{remark}
\begin{remark}
The averaged counterpart of Theorem~\ref{theo:LN_x_om} was originally proven in \cite{daiPra96} in the case of bounded disorder and extended to disorder with exponential moments in \cite{Luconthesis}, Th. 2.11. The corresponding averaged rate function is given by
\begin{equation}
\label{eq:averaged_rate_function}
\mathcal{ G}^{ ave}:= \lambda \mapsto \begin{cases}
\mathcal{H}(\lambda\vert  P^{\lambda}_{ \mu})= \int_{ \mathbb{ R}}\mathcal{ H}( \lambda^{ \omega} \vert P^{ \lambda, \omega}) \lambda_{ 2}(\dd \omega) + \mathcal{ H}( \lambda_{ 2} \vert \mu)& \text{ if } \lambda_{ 2} \ll \mu,\\
+\infty & \text{ otherwise}.
\end{cases}
\end{equation} 
Note that in particular
\begin{equation}
\label{eq:comp_qu_av}
0\leq \mathcal{ G}^{ ave}(\cdot) \leq \mathcal{ G}^{ que}(\cdot).
\end{equation}
\end{remark}
\begin{remark}
\label{rem:quenched_LLN}
As an easy corollary of Theorem~\ref{theo:LN_x_om} and Proposition~\ref{prop:mckean_vlasov}, we deduce a quenched convergence result for the empirical measure \eqref{eq:LN_proc} $L_{ N}^{ \underline{  \omega}}$ to the McKean-Vlasov solution $ \lambda^{ \ast}$, as $N\to \infty$, on any bounded time interval $[0, T]$. Note that a similar result was originally proven in \cite{Lucon2011}, under slightly weaker assumptions on the measure $ \mu$.
\end{remark}
\subsubsection{ Large deviations of the empirical flow}

The second result addresses the behavior of the empirical flow $ \mathscr{ L}_{ N}^{ \underline{  \omega}}$ defined in \eqref{eq:LN_flow}. Let $\mathcal{ C}^{ \infty}_{ 0}(]0, T[\times \mathbb{ R})$ be the space of infinitely differentiable functions with compact support on $]0, T[ \times \mathbb{ R}$. Let $ \mathbb{ A}$ bet the set of all flows $(q_{ t})_{ t\in [0, T]}$ satisfying $q_{ 2}= \mu$ and $ t \mapsto q_{ t}^{ \omega}$ is weakly differentiable for $ \mu$-almost every $ \omega$. Define also 
\begin{equation}
\label{eq:q_zero}
q_{ 0}(\dd x) := \int_{ \mathbb{ R}} q_{ 0}^{ \omega}(\dd x) \mu(\dd \omega).
\end{equation}
The main result of the section is 
\begin{theorem}
\label{theo:tilde_LN_x_om}
Under the assumptions of Section~\ref{sec:model_assumptions}, for every fixed sequence $ \left\lbrace \omega_{ i}\right\rbrace_{ i\geq1}$ satisfying \eqref{eq:convergence_mes_emp_mu}, \eqref{eq:conv_ell} and \eqref{eq:moment_cond_mu}, the sequence $ \left\lbrace \mathbf{ p}_{ N}^{ \underline{ \omega}}\right\rbrace_{ N\geq1}$ satisfies a strong large deviation principle, in $ \mathcal{ C}([0, T], \mathcal{ M}(\mathbb{ R} \times \mathbb{ R}))$ with speed $N$ governed by the good rate function
\begin{equation}
\label{eq:rate_function_mscG}
\mathscr{ G}:= q \mapsto \begin{cases} \mathcal{ H}(q_{ 0} \vert \gamma) + \int_{ \mathbb{ R}} \mathscr{ K}(q, \omega) \mu(\dd \omega),& \text{ if } q\in \mathbb{ A},\\
+\infty& \text{ otherwise.}
\end{cases}
\end{equation}
where
\begin{equation}
\label{eq:def_mscK}
\mathscr{ K}(q, \omega):= \sup_{ \phi\in \mathcal{ C}_{ 0}^{ \infty}(]0, T[\times \mathbb{ R})} \left\lbrace \int_{ 0}^{T} \left\langle \phi(t, \cdot)\, ,\, \partial_{ t} q_{ t}^{ \omega} - \mathscr{ L}^{ \omega}q_{ t}^{ \omega} \right\rangle\dd t - \int_{ 0}^{T} \left\langle (\partial_{ x} \phi(t, \cdot))^{ 2}\, ,\, q_{ t}^{ \omega}\right\rangle\dd t\right\rbrace.
\end{equation}
\end{theorem}
Theorem~\ref{theo:tilde_LN_x_om} is proven in Section~\ref{sec:proof_flow}.
\begin{remark}
Note that the present results can be generalized to the case of multidimensional diffusions (that is $ \mathcal{ E}= \mathbb{ R}^{ p}$ and $ \mathcal{ F}= \mathbb{ R}^{ q}$, for some $p, q\geq1$), up to minor changes in the notations. Nonetheless, we restrict to this simple case for simplicity.
\end{remark}
\subsection{Global and local empirical measures}
As a corollary of Theorems~\ref{theo:LN_x_om} and~\ref{theo:tilde_LN_x_om}, it is possible to deduce similar estimates for byproducts of $L_{ N}^{ \underline{ \omega}}$ and $ \mathscr{ L}_{ N}^{ \underline{ \omega}}$. 
\subsubsection{Global empirical measure}
First define
\begin{definition}
The empirical measure on the particles only is given by
\begin{align}
 \tilde{ L}_{ N}^{ \underline{ u}}:\mathcal{ E}^{ N} & \to \mathcal{ M}(\mathcal{ E}) \nonumber\\
\underline{ x} &\mapsto \tilde{ L}_{ N}^{ \underline{u}}[ \underline{ x}]:= \frac{ 1}{ N} \sum_{ i=1}^{ N} \delta_{x_{ i}}, \label{eq:LN_proc_single}
\end{align}
and the corresponding empirical flow by
\begin{equation}
\tilde{ \mathscr{ L}}_{ N}^{ \underline{u}}[ \underline{ x}]:= \left(t \mapsto\frac{ 1}{ N} \sum_{ i=1}^{ N} \delta_{x_{ i, t}}\right),\label{eq:LN_flow_single}
\end{equation}
Denote by 
\begin{align}
\tilde{\mathbf{ P}}_{ N}^{ \underline{u}}&:= P_{ N}^{ \underline{u}}\circ \left(\tilde{ L}_{ N}^{ \underline{u}}\right)^{ -1},\label{eq:PuN_LN_single}\\
\tilde{\mathbf{ p}}_{ N}^{ \underline{u}}&:= P_{ N}^{ \underline{u}}\circ \left(\tilde{ \mathscr{ L}}_{ N}^{ \underline{u}}\right)^{ -1},\label{eq:PuN_LN_flow_single}
\end{align} 
the respective laws of the empirical measure and flow of the coupled process \eqref{eq:edsi}.
\end{definition}
Proposition~\ref{prop:LN_x} and Proposition~\ref{prop:LN_x_flow} below concern respectively quenched LDPs for the empirical measures $ \tilde{ L}_{ N}^{ \underline{ \omega}}$ \eqref{eq:LN_proc_single} and the empirical flows $ \tilde{ \mathscr{ L}}_{ N}^{ \underline{ \omega}}$ \eqref{eq:LN_flow_single}:
\begin{proposition}
\label{prop:LN_x}
Under the assumptions of Section~\ref{sec:model_assumptions}, for any fixed sequence of disorder $ \left\lbrace \omega_{ i}\right\rbrace_{ i\geq1}$ satisfying \eqref{eq:convergence_mes_emp_mu}, \eqref{eq:conv_ell} and \eqref{eq:moment_cond_mu}, the sequence $ \left\lbrace \tilde{\mathbf{ P}}_{ N}^{ \underline{ \omega}}\right\rbrace_{ N\geq1}$ satisfies a strong large deviation principle in $ \mathcal{ M}( \mathcal{ E})$ with speed $N$, governed by the good rate function
\begin{equation}
\tilde{\mathcal{ G}}:= l\in \mathcal{ M}( \mathcal{ E}) \mapsto \inf \left\lbrace \int_{ \mathbb{ R}} \mathcal{ H}( \lambda^{ \omega} \vert P^{ \lambda, \omega}) \mu(\dd \omega),\ \lambda\in \mathcal{ Y}, \lambda_{ 2}= \mu, \int_{ \mathbb{ R}} \lambda^{ \omega}(\dd x) \mu(\dd \omega) = l(\dd x)\right\rbrace
\end{equation}
Moreover, $ \tilde{\mathcal{ G}}(\cdot)$ has a unique zero $ l^{ \ast}$, given by
\begin{equation}
\label{eq:l_ast}
l^{ \ast}(\dd x):= \int_{ \mathbb{ R}} \lambda^{ \ast, \omega}(\dd x) \mu(\dd \omega),
\end{equation}
where $ \lambda^{ \ast}$ is the unique zero of $ \mathcal{ G}(\cdot)$ given by Proposition~\ref{prop:mckean_vlasov}.
\end{proposition}
\begin{proposition}
\label{prop:LN_x_flow}
Under the assumptions of Section~\ref{sec:model_assumptions}, for any fixed sequence of disorder $ \left\lbrace \omega_{ i}\right\rbrace_{ i\geq1}$ satisfying \eqref{eq:convergence_mes_emp_mu}, \eqref{eq:conv_ell} and \eqref{eq:moment_cond_mu}, the sequence $ \left\lbrace\tilde{\mathbf{ p}}_{ N}^{ \underline{ \omega}}\right\rbrace_{ N\geq1}$ satisfies a large deviation principle in $ \mathcal{ C}([0, T], \mathbb{ R})$ with speed $N$, governed by the good rate function
\begin{equation}
\tilde{ \mathscr{ G}}:= p \mapsto \inf \left\lbrace \mathscr{ G}(q), p_{ t}(\dd x)= q_{ t}^{ \omega}(\dd x) \mu(\dd \omega),\ t\in [0, T]\right\rbrace.
\end{equation}
The rate function $ \tilde{ \mathscr{ G}}$ has a unique zero $p^{ \ast}$, given by
\begin{equation}
p^{ \ast}_{ t}(\dd x)= \lambda_{ t}^{ \ast, \omega}(\dd x) \mu(\dd \omega),\ t\in[0, T],
\end{equation}
where $ \lambda^{ \ast}$ is given by Proposition~\ref{prop:mckean_vlasov}.
\end{proposition}
\subsubsection{Local empirical measures}
Restrict in this paragraph to the situation where the disorder $ \underline{ u}$ takes its values in a finite set. The presence of a quenched disorder in \eqref{eq:edsi} breaks the exchangeability of the population $\{x_{ i}\}_{ i=1, \ldots, N}$. In particular, the knowledge of the whole empirical measure $L_{ N}^{ \underline{ \omega}}$ in \eqref{eq:LN_proc} is in general not sufficient to understand the behavior of one single particle $x_{ i}$. In this perspective, it is critical to consider the empirical measures restricted to the particles sharing the same disorder. 

 Consider for simplicity the situation where each $ u_{ i}$ belongs to $\{\pm1\}$, but (up to notational complications) the problem can easily be extended to the general case where the disorder takes its values in $\{ \epsilon_{ 1}, \ldots, \epsilon_{ p}\}\in \mathbb{ R}^{ p}$, for any $p\geq1$. In particular, there exists $ \alpha\in (0, 1)$ such that
\begin{equation}
\label{eq:mu_alpha}
\mu= (1- \alpha) \delta_{ -1} + \alpha \delta_{ +1}.
\end{equation}
For fixed $N\geq1$, the inhomogeneity coming from the disorder reduces to the respective sizes $N^{ +}:= \# \{i=1, \ldots, N,\ u_{ i}=+1\}$ and $N^{ -}:=\# \{i=1, \ldots, N,\ u_{ i}=-1\}$ of the particles with local disorder $+1$ and $-1$. Obviously $N= N^{ +}+ N^{ -}$. In this framework, any $ m(\dd x, \dd \omega)\in\mathcal{ Y}$ can be identified with $(m^{ +}(\dd x), m^{ -}(\dd x))\in \mathcal{ M}(\mathcal{ E})\times \mathcal{ M}(\mathcal{ E})$ where, for any bounded and continuous test function $ \phi$
\begin{align}
\int_{ \mathbb{ R}} \phi(x)m^{ +}(\dd x) &:= \frac{1}{ m( \mathcal{ E}\times\{+1\})}\int_{ \mathcal{ E}\times \mathcal{ F}} \phi(x) \mathbf{ 1}_{\{+1\}}(\omega) m(\dd x, \dd \omega),\\
\int_{ \mathbb{ R}} \phi(x)m^{ -}(\dd x) &:= \frac{1}{ m( \mathcal{ E}\times\{-1\})}\int_{ \mathcal{ E}\times \mathcal{ F}} \phi(x) \mathbf{ 1}_{\{-1\}}(\omega) m(\dd x, \dd \omega).
\end{align}
In this set-up, we define the empirical measures restricted to the local populations:
\begin{definition}
For any $ \epsilon\in\{\pm1\}$, define the local empirical measure restricted to the $ \epsilon$-population by
\begin{align}
L_{ N}^{\underline{ u}, \epsilon}:\mathcal{ E}^{ N} & \to \mathcal{ M}(\mathcal{ E}) \nonumber\\
\underline{ x} &\mapsto L_{ N}^{ \underline{u}, \epsilon}[ \underline{ x}]:= \frac{ 1}{ N^{ \epsilon}} \sum_{ \substack{i=1, \ldots, N,\\ u_{ i}= \epsilon}} \delta_{x_{ i}}. \label{eq:LN_proc_pm}
\end{align}
and the corresponding empirical flow by
\begin{equation}
\mathscr{ L}_{ N}^{ \underline{u}, \epsilon}[ \underline{ x}]:= \left(t \mapsto \frac{ 1}{ N^{ \epsilon}} \sum_{ \substack{i=1, \ldots, N,\\ u_{ i}= \epsilon}} \delta_{x_{ i, t}}\right),\ \epsilon\in\{\pm 1\}, \label{eq:LN_flow_pm}
\end{equation}
For any $ \epsilon\in \{\pm 1\}$, denote by
\begin{align}
\mathbf{ M}_{ N}^{ \underline{u}, \epsilon}&:= P_{ N}^{ \underline{u}}\circ \left(L_{ N}^{ \underline{u}, \epsilon}\right)^{ -1},\label{eq:PuN_LN_eps}\\
\mathbf{ m}_{ N}^{ \underline{u}, \epsilon}&:= P_{ N}^{ \underline{u}}\circ \left(\mathscr{ L}_{ N}^{ \underline{u}, \epsilon}\right)^{ -1},\label{eq:PuN_LN_eps_flow}
\end{align} 
the law of the local empirical measure and flow of the coupled processes \eqref{eq:edsi}.
\end{definition}
\begin{remark}
Of course, the definition of $(m^{ -}, m^{ +})$ only makes sense for measures $m$ such that $m(\mathcal{ E}\times\{ \epsilon\})>0$ for $ \epsilon\in \{\pm1\}$. This is not restrictive here: consider the subset of $ \mathcal{ Y}$ 
\begin{equation}
\mathcal{ Y}_{ p}:= \left\lbrace m\in \mathcal{ Y},\ m(\mathcal{ E}\times\{+1\})\geq \frac{ \alpha}{ 2}\ \text{ and } m(\mathcal{ E}\times\{-1\})\geq \frac{ 1- \alpha}{ 2}\right\rbrace.
\end{equation}
Note that, since the deterministic sequence $\{ \omega_{ i}\}_{ i=1,2,\ldots}$ satisfies \eqref{eq:convergence_mes_emp_mu}, there exists a deterministic $N_{ 0}$ such that for all $N\geq N_{ 0}$, almost surely, $(L_{ N}^{ \underline{ \omega}, +}, L_{ N}^{ \underline{ \omega}, -})$ belongs to $ \mathcal{ Y}_{ p}$. The set $ \mathcal{ Y}_{ p}$ is a subset of $ \mathcal{ Y}$, closed w.r.t. the topology of weak convergence, and the functional $ m \mapsto (m^{ +}, m^{ -})$ restricted to $ \mathcal{ Y}_{ p}$ is continuous. 
\end{remark}
Proposition~\ref{prop:LN_tilde} and Proposition~\ref{prop:LN_tilde_flow} below address the behavior of the local empirical measure $ L_{ N}^{ \underline{ \omega}, \epsilon}$ \eqref{eq:LN_proc_pm} and flow $ \mathscr{ L}_{ N}^{ \underline{ \omega}, \epsilon}$ \eqref{eq:LN_flow_pm}:
\begin{proposition}
\label{prop:LN_tilde}
For any fixed sequence of disorder $\{\omega_{ i}\}_{ i\geq1}$ satisfying the assumptions of Section~\ref{sec:model_assumptions}, the sequence $ \left\lbrace\mathbf{ M}_{ N}^{ \underline{ \omega}, +1}\right\rbrace_{ N\geq1}$ (resp. $ \left\lbrace\mathbf{ M}_{ N}^{ \underline{ \omega}, -1}\right\rbrace_{ N\geq1}$) satisfies on $ \mathcal{ M}( \mathcal{ E})$ a large deviation principle with speed $N$, governed by the good rate function $ \bar{ \mathcal{ G}}_{ +1}$ (resp. $ \bar{ \mathcal{ G}}_{ -1}$) given by
\begin{align}
\bar{\mathcal{ G}}_{+1}&:= l \mapsto \inf_{ \lambda \in \mathcal{ M}( \mathcal{ E})} \left\lbrace (1- \alpha) \mathcal{ H}( \lambda \vert P^{ (1- \alpha) \lambda + \alpha l, -1}) + \alpha \mathcal{ H}( l\vert P^{ (1- \alpha) \lambda + \alpha l, +1})\right\rbrace,\label{eq:G_bar_flow}\\
\bar{\mathcal{ G}}_{-1}&:= l \mapsto \inf_{ \lambda \in \mathcal{ M}( \mathcal{ E})} \left\lbrace (1- \alpha) \mathcal{ H}(l\vert P^{ (1- \alpha) l + \alpha \lambda, -1}) + \alpha \mathcal{ H}( \lambda\vert P^{ (1- \alpha) l + \alpha \lambda, +1})\right\rbrace.\label{eq:G_bar_flow_minus}
\end{align}
\end{proposition}
\begin{proposition}
\label{prop:LN_tilde_flow}
For any fixed sequence of disorder $ \left\lbrace\omega_{ i}\right\rbrace_{ i\geq1}$ satisfying the assumptions of Section~\ref{sec:model_assumptions}, the sequence $ \left\lbrace\mathbf{ m}_{ N}^{ \underline{ \omega}, +1}\right\rbrace_{ N\geq1}$ (resp. $ \left\lbrace\mathbf{ m}_{ N}^{ \underline{ \omega}, -1}\right\rbrace_{ N\geq1}$) satisfies a large deviation principle in $ \mathcal{ C}([0, T], \mathcal{ M}(\mathbb{ R}))$ with speed $N$, governed by the good rate function $\bar{ \mathscr{ G}}_{ +1}$ (resp. $\bar{ \mathscr{ G}}_{ -1}$) given by
\begin{align}
\bar{ \mathscr{ G}}_{ +1}&:= p \mapsto \inf_{ q\in \mathcal{ C}([0, T], \mathbb{ R})} \mathscr{ G}(q, p),\\
\bar{ \mathscr{ G}}_{ -1}&:= p \mapsto \inf_{ q\in \mathcal{ C}([0, T], \mathbb{ R})} \mathscr{ G}(p, q).
\end{align}
\end{proposition}
Propositions~\ref{prop:LN_x},~\ref{prop:LN_x_flow},~\ref{prop:LN_tilde},~\ref{prop:LN_tilde_flow} are straightforward consequences of Theorems~\ref{theo:LN_x_om} and~\ref{theo:tilde_LN_x_om} and the contraction principle. We leave the proofs to the reader.
\subsection{ Links with the existing litterature and perspectives}
\subsubsection{Large population behavior of mean-field systems}
Systems of mean-field interacting diffusions and their relations to nonlinear PDEs of McKean-Vlasov type have been studied in numerous situations. Concerning large deviations results, the paper of Dawson and G\"artner \cite{MR885876} (see also \cite{daiPra96,MR1331217} and references therein) addresses large deviations of the empirical flow for homogeneous mean-field systems, using projective limits arguments. Similar techniques have been recently adapted to the case of diffusions with spatial interactions \cite{Muller:2015aa}. The methods introduced in \cite{MR885876} allow to consider a larger class of diffusions than we do here (with weaker regularity and  unbounded coefficients). It is likely that a similar approach would still be applicable to our case, but we keep the assumptions of Section~\ref{sec:model_assumptions} for the simplicity of exposition. Moreover, although this work uses critical ideas from \cite{MR885876}, it differs in the sense that we do not only study the behavior of the empirical flow, but also the empirical measure, which is a more general object.  
\subsubsection{Large deviations for disordered systems}
The question of quenched large deviations for disordered particle systems has been addressed in various contexts. Examples on the subject are: random walks in random environment \cite{MR1785454,MR2779403},  Gibbs random fields \cite{MR976534}, random projections on $\ell^{ p}$-balls \cite{Bovier:2014aa, Gantert:2015ab}, interacting diffusions \cite{MR1307957,MR2248897}, regular conditional probabilities \cite{Zuijlen:2016aa}.

As mentioned earlier in the introduction, one of the first papers to address inhomogeneous mean-field diffusions is \cite{daiPra96}, where \emph{averaged large deviations} are proven in the case of bounded coefficients. This result was extended to unbounded disorder in \cite{Luconthesis}. As a corollary, averaged law of large numbers and fluctuations are derived in \cite{daiPra96}, using in particular techniques from Bolthausen \cite{Bolthausen1986}. Averaged moderate large deviations for Kuramoto-type oscillators are addressed in \cite{thesis_dahms}.

The present work addresses the quenched counterpart of the large deviations result of \cite{daiPra96}, that is when the disorder sequence is frozen once and for all. Note that we allow the coefficients of the diffusion to be possibly unbounded w.r.t. the disorder (as it is the case in the Kuramoto model). We also capture here the possible dependence of the initial condition in the random environment (recall \eqref{eq:initial_condition_edsN}). A quenched law of large numbers result easily follows (Remark~\ref{rem:quenched_LLN}) from the quenched large deviations estimate, by a usual uniqueness result of the zeros of the rate function. 

Note that, contrary to \cite{daiPra96}, it is very unlikely that one could derive quenched fluctuations from the results of the present paper, for the simple reason that quenched fluctuations do not hold in general: consider the simple case of \eqref{eq:edsi} where $g(x, \omega) = \omega$ and $f \equiv 0$, that is $ \dd x_{ i, t}= \omega_{ i} \dd t + \dd B_{ i, t}$, $N$ independent (but not identically distributed) Brownian motions with different drifts. Considering the fluctuations of the empirical measure of such model requires to study the convergence of functionals of the type $ \sqrt{N} \left( \frac{ 1}{ N} \sum_{ i=1}^{ N} \phi(\omega_{ i}) - \mathbb{ E}( \phi(\omega))\right)$. For a fixed realization of $ \left\lbrace\omega_{ i}\right\rbrace_{ i\geq1}$ and a nontrivial $ \phi$, this quantity does not converge (it only converges in law w.r.t. $ \omega$). A reasonable notion of quenched fluctuations for this class of models is more intricate and has already been addressed in \cite{Lucon2011}. In any case, it is doubtful that the fluctuation results of \cite{Lucon2011} could be derived from the results of the present paper. 
\subsubsection{Non-exchangeability and propagation of chaos}
A well-known result \cite{SznitSflour} about mean-field systems tells that when the particles are exchangeable (in the non-disordered or averaged cases for example) the \emph{macroscopic} convergence of the empirical measure of the particles to the McKean-Vlasov PDE \eqref{eq:mckean_vlasov} is equivalent to the \emph{microscopic} notion of propagation of chaos, that is the convergence in law of any $(x_{ 1}, \ldots, x_{ k})$ to $k$ independent copies of the nonlinear process $ \bar x$ (see Remark~\ref{rem:McKV_nonlinear}), whose law solves \eqref{eq:mckean_vlasov}.

The influence of a quenched disorder on the macroscopic dynamics of mean-field particles systems has been particularly studied in the case of the Kuramoto model (see \cite{MR3207725,Bertini:2013aa,Lucon:2015aa} and references therein). The Kuramoto model without disorder exhibits a phase transition \cite{Acebron2005}: when the strength of interaction $K$ is smaller than a critical value $K_{ c}$, the McKean-Vlasov PDE \eqref{eq:mckean_vlasov} admits only the flat solution $ \frac{ 1}{ 2\pi}$ as a stationary solution (incoherence), whereas there is at least one nontrivial stationary solution $q(\cdot)$ when $K\geq K_{ c}$ (synchrony). By invariance by rotation of $-K \cos(\cdot)$, this synchronized state actually generates a whole circle of invariant states $ \mathscr{ C}:=\{q(\cdot - \psi),\ \psi\in \mathbb{ S}\}$. If one had to sum-up in one word the conclusions of \cite{MR3207725,Bertini:2013aa,Lucon:2015aa}, the main message would be that adding a quenched disorder to the Kuramoto model induces an asymmetry in the system which generates macroscopic traveling-waves along a perturbed manifold $ \mathscr{ C}^{ pert}\approx \mathscr{ C}$. The speed and direction of the traveling waves, as well as the time-scale at which they can be observed depend on the nature of the asymmetry induced by the disorder (see \cite{MR3207725,Bertini:2013aa,Lucon:2015aa} for precise statements).

In our quenched set-up, the exchangeability assumption of the particles is obviously not satisfied. A natural question is the following: is it possible to understand the behavior of one single particle w.r.t. the macroscopic behavior of the whole empirical measure? In the case of a binary disorder $ \omega\in \{\pm1\}$, with a majority of $+1$, the macroscopic behavior of the system is a traveling wave in the $+1$ direction \cite{MR3207725,Lucon:2015aa}. In this context, what can we say about an atypical behavior of a particle $ x^{ -}$ with local disorder $ \omega= -1$? Can we measure the difficulty for this particle of going in the direction opposite to the majority? A first step in this direction could be to derive from Proposition~\ref{prop:LN_tilde} Gibbs conditioning principle (\cite{Dembo1998}, p.~323) for the law of $x^{ -}$ conditioned on the behavior of $L_{ N}^{ +}$. This would require in particular a better understanding of the rate function $ \bar{\mathcal{ G}}$ defined in \eqref{eq:G_bar_flow} in the case of the Kuramoto model.
\subsection{ Outline of the paper}
The proof of Theorem~\ref{theo:LN_x_om} is divided into several steps. Section~\ref{sec:abstract_sanov} states a Sanov-type result in the case the interaction between particles is removed. Section~\ref{sec:large_deviations_with_interaction} establishes the full large deviation result for a rate function that is identified with $ \mathcal{ G}$ in Section~\ref{sec:indentification_rate_function}. The proof of the large deviations of the empirical flow (Theorem~\ref{theo:tilde_LN_x_om}) is given in Section~\ref{sec:proof_flow}. Finally, Appendix~\ref{sec:proof_Sanov} contains the proof of the abstract Sanov theorem used in Section~\ref{sec:abstract_sanov}.
\section{Sanov theorem for independent diffusions}
\label{sec:abstract_sanov}
For any generic disorder sequence $ \left\lbrace u_{ i}\right\rbrace_{ i\geq1}$, the system \eqref{eq:edsi} under the disorder $ \underline{ u}=(u_{ 1}, \ldots, u_{ N})$ can be rewritten in a compact form:
 \begin{equation}
\label{eq:edsN} 
\underline{x}_{ t} = \underline{ x}_{ 0} + \int_{ 0}^{t} g( \underline{ x}_{ s}, \underline{ u})\dd s- \int_{ 0}^{t}\nabla{H_{N}}(\underline{x}_{ s}, \underline{u})\dd s + \underline{B}_{ t},
\end{equation}
where $ \underline{ x}:= (x_{ 1}, \ldots, x_{ N})$, $g( \underline{ x}, \underline{ u}):=(g(x_{ 1}, u_{ 1}), \ldots, g(x_{ N}, u_{ N}))$ and $\underline{B}:=(B_1, \dots, B_N)$. Introduce now the same system where the interaction has been removed:
\begin{equation}
\label{eq:edsdB}
\underline{x}_{ t} = \underline{ x}_{ 0} + \int_{ 0}^{t} g(\underline{x}_{ s}, \underline{u})\dd s + \underline{B}_{ t}.
\end{equation}
The solution $\underline{x}=(x_1, \dots, x_N)$ to \eqref{eq:edsdB} consists of $N$ independent (but not identically distributed since the disorder $ u_{ i}$ is different for each $x_{ i}$) copies of
\begin{equation}
\label{eq:diff_copy}
x_{ t}= x_{ 0} + \int_{ 0}^{t}g(x_{ s}, u) \dd s + b_{t},
\end{equation}where $b$ is a standard Brownian motion in $ \mathbb{ R}$. We denote by $W^{u}$ the law of \eqref{eq:diff_copy} so that the law of \eqref{eq:edsdB} is 
\begin{equation}
\label{eq:WuN}
W^{\underline{u}}_{N}(\dd \underline{ x}):= W^{u_{1}}(\dd x_{ 1})\times\dots\times W^{u_{N}}(\dd x_{ N}).
\end{equation}
Define 
\begin{equation}
\label{eq:WuN_LN}
\mathbf{ W}_{ N}^{ \underline{u}}:= W_{ N}^{ \underline{u}}\circ \left(L_{ N}^{ \underline{u}}\right)^{ -1}
\end{equation} 
as the law of the empirical measure $L_{ N}^{ \underline{u}}$ of the uncoupled diffusions \eqref{eq:edsdB}, in the environment $ \underline{u}$. The main result of this section is the following:
\begin{proposition}
\label{prop:LDP_uncoupled}
For any fixed disorder sequence $ \left\lbrace u_{ i}\right\rbrace_{ i\geq1}$ and any measure $ \nu \in \mathcal{ M}( \mathbb{ R})$ such that
\begin{equation}
\label{eq:hyp_ui}
 \frac{ 1}{ N} \sum_{ i=1}^{ N} \delta_{ u_{ i}} \to_{ N\to \infty} \nu, \text{ for the weak topology on }\mathcal{ M}( \mathbb{ R}),
\end{equation}
 the sequence $ \mathbf{ W}_{ N}^{ \underline{ u}}$ satisfies a large deviation principle in $ \mathcal{ Y}$, with speed $N$ and good rate function
\begin{align}
\mathcal{ I}_{ \nu}(\lambda):= \begin{cases}
 \mathcal{ H}(\lambda \vert W^{ u}\times \nu), &\text{if } \lambda_{ 2}( \dd u)= \nu( \dd u),\\
+\infty & \text{ otherwise},\label{eq:rate_I_no_inter}
\end{cases} 
\end{align}
where $ \mathcal{ H}(\cdot \vert \cdot)$ is the relative entropy defined in \eqref{eq:relative_entropy}.
\end{proposition}

\subsection{ Abstract quenched Sanov Theorem}
Proposition~\ref{prop:LDP_uncoupled} is a consequence of an abstract result valid in a general setting. The result stated in this paragraph appears in different forms in the literature (see for example L\'eonard \cite{2007arXiv0710.1461L}, Proposition~3.2, Cattiaux and Léonard \cite{MR1307957}, Dembo and Zeitouni \cite{Dembo1998}, Th.~6.2.10, or Comets \cite{MR976534}, Th.~III.1). We state the result (Proposition~\ref{prop:leonard}) in this paragraph and refer to Appendix~\ref{sec:proof_Sanov} where a sketch of proof is given for the sake of completeness. We refer to \cite{2007arXiv0710.1461L} or \cite{Dembo1998} for details.

Let $ \mathcal{ E}$ and $ \mathcal{ F}$ be two Polish spaces. Recall that in the setting of Section~\ref{sec:model_assumptions}, $\mathcal{ E}= \mathcal{ C}([0, T], \mathbb{ R})$ and $ \mathcal{ F}= \mathbb{ R}$ but the results of this section are true for any Polish spaces. Let $ \mathcal{ X}$ be the algebraic dual of $ \mathcal{ W}:= \mathcal{ C}_{ b}( \mathcal{ E}\times \mathcal{ F})$, the set of bounded and continuous functions on $ \mathcal{ E}\times \mathcal{ F}$. We equip $ \mathcal{ X}$ with the $\ast$-weak topology $ \sigma( \mathcal{ X}, \mathcal{ W})$ and $ \mathcal{ B}$ the corresponding Borel-$ \sigma$ field. Denote by $ \mathcal{ Y}:=\mathcal{ M}( \mathcal{ E}\times \mathcal{ F})$ the set of probability measures on $ \mathcal{ E}\times \mathcal{ F}$. We endow $ \mathcal{ Y}$ with the trace of the $\ast$-weak topology $ \sigma( \mathcal{ X}, \mathcal{ W})$ on $ \mathcal{ Y}$ and with the Borel $ \sigma$-field $ \mathcal{ B}$. Denote by $ \left\langle \lambda\, ,\, w\right\rangle$ the value of the linear functional $ \lambda\in \mathcal{ X}$ at point $ \phi\in \mathcal{ W}$. 

Let us fix a deterministic sequence $ \left\lbrace u_{ i}\right\rbrace_{ i\geq1} \in \mathcal{ F}^{ \mathbb{ N}}$ and a measure $ \nu\in \mathcal{ M}( \mathcal{ F})$ such that 
\begin{equation}
\label{eq:conv_nu_abstract}
 \frac{ 1}{ N} \sum_{ i=1}^{ N} \delta_{ u_{ i}} \to_{ N\to \infty} \nu, \text{ for the weak topology on }\mathcal{ M}( \mathcal{ F}),
\end{equation}
Fix also a measurable mapping $ \rho:\mathcal{ F} \to \mathcal{ M}( \mathcal{ E}), u \mapsto \rho^{ u}$ and let $ \left\lbrace x_{ i}\right\rbrace_{ i\geq1}$ a sequence of independent random variables in $ \mathcal{ E}$ such that for each $i\geq1$, $x_{ i}\sim \rho^{u_{ i}}$. We suppose that the application $ u \mapsto \rho^{ u}$ is Feller on $ \mathcal{ M}( \mathcal{ E})$.

Consider the quenched empirical measure $L_{ N}^{ \underline{ u}}:= \frac{ 1}{ N} \sum_{ i=1}^{ N} \delta_{ (x_{ i}, u_{ i})} \in \mathcal{ Y}$ and denote by $ \mathbf{ m}_{ N}^{ \underline{ u}}$ the law of $L_{ N}^{ \underline{ u}}$ induced on $ \mathcal{ Y}$ by the law $\bigotimes_{ i=1}^{ +\infty} \rho^{u_{ i}}$ of the variables $ \left\lbrace x_{ i}\right\rbrace_{ i\geq1}$. We have the following result:
\begin{proposition}[ (\cite{2007arXiv0710.1461L}, Proposition~3.2)]
\label{prop:leonard}
For any sequence $ \left\lbrace u_{ i}\right\rbrace_{ i\geq1}$ satisfying \eqref{eq:conv_nu_abstract}, the sequence $ \left\lbrace \mathbf{ m}_{ N}^{ \underline{ u}}\right\rbrace_{ N\geq1}$ satisfies a large deviation principle, at speed $N$, governed by the good rate function\begin{align}
\mathcal{ I}_{ \mu}: \mathcal{ Y} &\to [0, +\infty] \nonumber\\
 \lambda & \mapsto 
\begin{cases} \mathcal{ H}(\lambda \vert \rho^{u} \times \nu), & \text{ if } \lambda_{ 2}(d u) = \nu(d u),\\
+\infty & \text{ otherwise}.
\end{cases} \label{eq:gen_rate_I}
\end{align}
\end{proposition}
\begin{remark}
This result has been stated in \cite{2007arXiv0710.1461L} in the context of Monge-Kantorovich problems in optimal transport. A similar result may be found in \cite{MR1307957}, Th.~2.1 in the context of Nelson processes, conditioned on their initial values. Note that one could also see this result as an application of the contraction principle to Th.~III.1 in \cite{MR976534}. Proposition~\ref{prop:leonard} uses techniques of large deviations for projective limits introduced by Dawson and G\"artner \cite{MR885876} (see also Dembo and Zeitouni \cite{Dembo1998}, Th. 6.2.10).
\end{remark}
\subsection{Proof of Proposition~\ref{prop:LDP_uncoupled}}
Proposition~\ref{prop:LDP_uncoupled} follows from Proposition~\ref{prop:leonard} applied to the case $ \mathcal{ E}= \mathcal{ C}([0, T], \mathbb{ R})$, $ \mathcal{ F}= \mathbb{ R}$, $ \rho^{u}= W^{u}$ and $ \mathbf{ m}_{ N}^{ \underline{ u}}:= \mathbf{ W}_{ N}^{ \underline{ u}}$, once we have established the regularity of $ u \mapsto W^{ u}$:
\begin{lemma}
\label{lem:Wu_Feller}
Under the assumptions of Section~\ref{sec:model_assumptions}, the application $ \mathbb{ R} \to \mathcal{ M}( \mathcal{ E}),\ u \mapsto W^{ u}(\dd x)$ is Feller.
\end{lemma}
\begin{proof}[Proof of Lemma~\ref{lem:Wu_Feller}]
Recall that $W^{ u}(\dd x)$ is the law of the diffusion $x_{ t}= x_{ 0}^{ u} + \int_{ 0}^{t}g(x_{ s}, u) \dd s + b_{t}$ in the environment $u$, where the initial condition $x_{ 0}^{ u}$ is sampled according to $ \gamma^{ u}(\dd x)$. For any $u, v\in \mathbb{ R}$, any deterministic $ x_{ 0}, \tilde{ x}_{ 0}$, consider $ x$ (resp. $ \tilde{ x}$) solution of  \eqref{eq:diff_copy} with initial condition $x_{ 0}$ (resp. $ \tilde{ x}_{ 0}$), with disorder $u$ (resp. $v$). We assume that $x$ and $ \tilde{ x}$ are driven by the same Brownian motion $b$. By Ito Formula, for any $0\leq t\leq T$, using \eqref{eq:Lip_g_om} and \eqref{eq:Lip_g_x},
\begin{align*}
\left\vert x_{ t} - \tilde{ x}_{ t} \right\vert^{ 2} &=  \left\vert x_{ 0} - \tilde{ x}_{ 0} \right\vert^{ 2} + 2\int_{0}^{t} \left(g(x_{ s}, u) - g( \tilde{ x}_{ s}, v)\right) (x_{ s} - \tilde{ x}_{ s}) \dd s,\\
&\leq \left\vert x_{ 0} - \tilde{ x}_{ 0} \right\vert^{ 2} + \int_{0}^{t} \left\vert g(x_{ s}, u) - g( \tilde{ x}_{ s}, v)\right\vert^{ 2} \dd s + \int_{0}^{t}  \left\vert x_{ s} - \tilde{ x}_{ s} \right\vert^{ 2} \dd s,\\
&\leq \left\vert x_{ 0} - \tilde{ x}_{ 0} \right\vert^{ 2} + t b(u, v) +  a \int_{0}^{t} \left\vert x_{ s} - \tilde{ x}_{ s}\right\vert^{ 2} \dd s,
\end{align*}
where $a:= \left(2C_{ g}^{ 2} +1\right)$ and $ b(u, v):= 2C_{ g}^{ 2} \left\vert u- v \right\vert^{ 2} \left(1+ \left\vert u \right\vert^{ k_{ 2}} + \left\vert v \right\vert^{ k_{ 2}}\right)^{ 2}$.
Consequently, 
\begin{align*}
\sup_{ 0\leq t \leq T} \left\vert  x_{ t}- \tilde{ x}_{ t} \right\vert^{ 2} &\leq \left\vert x_{ 0} - \tilde{ x}_{ 0} \right\vert^{ 2} + T b(u, v) +  a \int_{0}^{T} \sup_{ r\leq s}\left\vert x_{ r} - \tilde{ x}_{ r}\right\vert^{ 2} \dd s.
\end{align*}
By Gronwall Lemma and taking the expectation w.r.t. to the Brownian motion $b$, one obtains
\begin{align}
\label{aux:cont_x_tildex}
\mathbf{ E}_{ b} \left( \sup_{ 0\leq t \leq T}\left\vert x_{ t} - \tilde{ x}_{ t} \right\vert^{ 2}\right) \leq \left( \left\vert x_{ 0} - \tilde{ x}_{ 0} \right\vert^{ 2} + T b(u, v)\right) \exp \left(aT\right).
\end{align}
This implies that, for any bounded continuous function $ \varphi$ on $ \mathcal{ E}$, the function $(x_{ 0}, u) \mapsto G(x_{ 0}, u):=\mathbf{ E}_{ b}^{ x_{ 0}, u}( \varphi(\cdot))$ is continuous (here $ \mathbf{ E}_{ b}^{ x_{ 0}, u}$ denotes the expectation w.r.t. $b$ conditioned on the initial value $x_{ 0}$, in the environment $u$). Indeed, for any $(x_{ 0}, u)\in \mathbb{ R}\times \mathbb{ R}$ and any sequence $ \left\lbrace (x_{ 0, n}, u_{ n})\right\rbrace_{ n\geq1}$ that converges to $(x_{ 0}, u)$ for $n\to\infty$, we know from \eqref{aux:cont_x_tildex} that the solution $ \left\lbrace x_{ n, t}\right\rbrace_{ t\in[0, T]}\in \mathcal{ E}$ starting from $x_{ 0, n}$ converges in $L^{ 2}$ to $x$. In particular, there exists a subsequence (that we rename $ \left\lbrace x_{ n, t}\right\rbrace_{ t\in [0, T]}$ for convenience) such that $ \left\lbrace x_{ n, t}\right\rbrace_{ t\in[0, T]}$ converges almost surely to $x$ in $ \mathcal{ E}$. By Fatou Lemma and the continuity of $G$,
\begin{align*}
 G(x_{ 0}, u)=\mathbf{ E}_{ b}^{ x_{ 0}, u}( \varphi(\cdot)) \leq \liminf_{ n\to \infty}\mathbf{ E}_{ b}^{ x_{ n, 0}, u_{ n}}( \varphi(\cdot)).
\end{align*}Applying the same estimate to $- \varphi$, we deduce the joint continuity of $(x_{ 0}, u) \mapsto G(x_{ 0}, u)$. For such a function $ \varphi$, we have, for any $u\in \mathbb{ R}$, $ \int_{ \mathcal{ E}} \varphi(x) W^{ u}(\dd x) = \int_{ \mathbb{ R}} G(x_{ 0}, u)\lambda^{ u}(\dd x_{ 0})$. Let $ \left\lbrace u_{ n}\right\rbrace_{ n\geq1}$ a sequence such that $u_{ n}\to u$ as $n\to\infty$. The measure $ \lambda^{ u}$ is tight: for all $ \varepsilon>0$, there exists $A>0$ such that $ \lambda^{ u}([-A, A]^{ c})\leq \varepsilon$. Let $ h_{ A}: \mathbb{ R}\to [0, 1]$ be a piecewise-linear function such that $ h_{ A}(x)= 0$ for $x\in[-A, A]$ and $ h_{ A}(x)=1$ for $x\in [-(A+1), A+1)]^{ c}$, so that $ \mathbf{ 1}_{ [-(A+1), A+1]^{ c}} \leq h_{ A}\leq \mathbf{ 1}_{ [-A, A]^{ c}}$ and $ \int_{ \mathbb{ R}} h_{ A}(x) \lambda^{ u}(\dd x)\leq \varepsilon$. Since $ h_{ A}$ is continuous and $ u \mapsto \lambda^{ u}$ is Feller, there exists $n_{ 0}\geq1$ such that for all $n\geq n_{ 0}$, $ \int_{ \mathbb{ R}} h_{ A}(x) \lambda^{ u_{ n}}(\dd x) \leq 2 \varepsilon$. Hence, for all $n\geq n_{ 0}$,
\begin{equation}
\label{aux:control_lambda_un}
\lambda^{ u_{ n}}([-(A+1), A+1]^{ c})\leq 2 \varepsilon.
\end{equation}
For such $n$,
\begin{align*}
\left\vert \int_{ \mathbb{ R}} G(x_{ 0}, u_{ n}) \lambda^{ u_{ n}}(\dd x_{ 0}) - \int_{ \mathbb{ R}} G(x_{ 0}, u) \lambda^{ u}(\dd x_{0})\right\vert&\leq \left\vert \int_{ \mathbb{ R}} \left(G(x_{ 0}, u_{ n}) - G(x_{ 0}, u)\right) \lambda^{ u_{ n}}(\dd x_{ 0})\right\vert\\
& + \left\vert \int_{ \mathbb{ R}} G(x_{ 0}, u) \lambda^{ u_{ n}}(\dd x_{ 0}) - \int_{ \mathbb{ R}} G(x_{ 0}, u) \lambda^{ u}(\dd x_{ 0}) \right\vert.
\end{align*}
Since $G$ is continuous and $ u \mapsto \lambda^{ u}$ is Feller, the second term above obviously goes to $0$ as $n\to \infty$. We focus on the first term: for $A$ and $n\geq n_{ 0}$ defined above, using \eqref{aux:control_lambda_un}
\begin{align}
\left\vert \int_{ \mathbb{ R}} \left(G(x_{ 0}, u_{ n}) - G(x_{ 0}, u)\right) \lambda^{ u_{ n}}(\dd x_{ 0})\right\vert &\leq \int_{-(A+1)}^{ A+1} \left\vert G(x_{ 0}, u_{ n}) - G(x_{ 0}, u) \right\vert \lambda^{ u_{ n}}(\dd x_{ 0}) \nonumber\\
&+ \int_{[-(A+1), A+1]^{ c}} \left\vert G(x_{ 0}, u_{ n}) - G(x_{ 0}, u) \right\vert \lambda^{ u_{ n}}(\dd x_{ 0}),\nonumber\\
& \leq \int_{-(A+1)}^{ A+1} \left\vert G(x_{ 0}, u_{ n}) - G(x_{ 0}, u) \right\vert \lambda^{ u_{ n}}(\dd x_{ 0})+ 4 \varepsilon \left\Vert \varphi \right\Vert_{ \infty} \label{aux:bound_int_G}
\end{align}
The function $G$ is continuous on the compact $K:=[-(A+1), A+1]\times [u-1, u+1]$, and hence, uniformly continuous: there exists $ \eta>0$ such that if $(x, u),\ (y, v)\in K$ are such that $ \left\vert x-y \right\vert + \left\vert u-v \right\vert< \eta$, $ \left\vert G(x, u) - G(y, v) \right\vert \leq \varepsilon$. Taking $n$ sufficiently large such that $ \left\vert u_{ n} - u \right\vert < \eta \wedge 1$, one can bound the first term in \eqref{aux:bound_int_G} by $ \varepsilon$. Thus, $\int_{ \mathbb{ R}} \varphi(x) W^{ u_{ n}}(\dd x) \to \int_{ \mathbb{ R}} \varphi(x) W^{ u}(\dd x)$ for any bounded continuous test function $ \varphi$. Lemma~\ref{lem:Wu_Feller} is proven.
\end{proof}
\section{Large deviation for the process with interaction}
\label{sec:large_deviations_with_interaction}
We now derive a large deviation principle for the empirical measure of the full process with interaction \eqref{eq:edsN}, based on usual Girsanov transform. Since the sequence of disorder $ \left\lbrace \omega_{ i}\right\rbrace_{ i\geq1}$ defined in Section~\ref{sec:model_assumptions} may not be bounded, usual techniques based on Varadhan's Lemma do not apply directly. Thus, we apply a truncation procedure to the sequence of disorder and first derive the large deviation result for the truncated sequence (Section~\ref{sec:suppcomp}). The main point is then to take the limit as the bound on the disorder goes to $\infty$, through an exponentially good coupling (Section~\ref{sec:exp_coupling}).

\subsection{Application of Girsanov's theorem}
We first compute the Radon-Nykodym derivative between the case with both interaction and disorder \eqref{eq:edsN} and the case without interaction \eqref{eq:edsdB}. 
\begin{proposition}
\label{prop:G1} For any sequence of disorder $ \underline{ u}=(u_{ 1}, \ldots, u_{ N})\in \mathcal{ F}^{ N}$,
\begin{equation}
\label{eq:G1}
\frac{\dd P_{N}^{\underline{u}}}{\dd W^{\underline{u}}_{N}} = \exp\left(N \mathcal{J}(L_{N}^{ \underline{ u}}) - \mathcal{K}(L_N^{ \underline{ u}})\right),
\end{equation}
where, for $ \lambda\in \mathcal{Y}$,
\begin{align}
\mathcal{K}(\lambda)&:= \frac12 \int \partial_{ x}^{ 2}f(0, \omega, \omega) \lambda(\dd x, \dd \omega),\label{eq:mcK}\\
\mathcal{J}(\lambda) &:= \mathcal{J}^{ (1)}(\lambda) + \mathcal{J}^{ (2)}(\lambda) + \mathcal{J}^{ (3)}(\lambda) + \mathcal{J}^{ (4)}(\lambda),\label{eq:mcJ}
\end{align}
where, 
\begin{align*}
\mathcal{J}^{ (1)}(\lambda) &:= \frac{1}{2}\int_{ (\mathcal{ E}\times \mathcal{ F})^{ 2}} \left[f(x_{0}-\tilde x_{0}, \omega, \tilde\omega) -f(x_{T}-\tilde x_{T}, \omega,\tilde \omega)\right]\lambda(\dd x, \dd\omega) \lambda(\dd\tilde x, \dd\tilde\omega),\\
\mathcal{J}^{ (2)}(\lambda) &:= \int_{0}^{T}\int_{ (\mathcal{ E}\times \mathcal{ F})^{ 2}} \partial_{ x}f(x_{s}-\tilde x_{s}, \omega, \tilde\omega)g(x_s, \omega)\lambda(\dd x, \dd\omega)  \lambda(\dd \tilde x, \dd\tilde\omega)\dd s,\\
\mathcal{J}^{ (3)}(\lambda) &:= \frac{1}{2} \int_{0}^{T} \int_{ (\mathcal{ E}\times \mathcal{ F})^{ 2}} \partial_{ x}^{ 2}f(x_{s}-\tilde x_{s}, \omega, \tilde\omega)\lambda(\dd x, \dd\omega) \lambda(\dd \tilde x, \dd\tilde\omega)\dd s,\\
\mathcal{J}^{ (4)}(\lambda) &:= -\frac{1}{2}\int_{0}^{T}\int_{ \mathcal{ E}\times \mathcal{ F}}   \left(\int_{ \mathcal{ E}\times \mathcal{ F}} \partial_{ x}f(x_{s}-\tilde x_{s}, \omega, \tilde\omega) \lambda(\dd\tilde x, \dd\tilde\omega)\right)^{2}\lambda(\dd x, \dd\omega)\dd s.
\end{align*}
\end{proposition}
\begin{proof}[Proof of Proposition~\ref{prop:G1}]
We follow here closely the calculations of \cite{daiPra96}. Let us fix $\underline{u}$ and $\underline{B}$ and consider $\underline{x}$ the unique solution to \eqref{eq:edsdB}. Let $ \mathcal{ F}_{ t}=\sigma(\underline{B}_{ s}, s\leq t)$ be the filtration generated by the Brownian motion $\underline{B}$. By an application Girsanov transform, one obtains, for all $0\leq t\leq T$ that
\begin{align*}
\frac{\dd P_{N}^{\underline{u}}}{\dd W^{\underline{u}}_{N}}_{\vert \mathcal{ F}_{ t}} &= \exp\left(M_{t} - \frac{1}{2} \left\langle M\right\rangle_{ t}\right),
\end{align*}
for
\begin{equation*}
 M_{t}:= \int_{0}^{t}{-\nabla H_{N}(\underline{x}_{ s}, \underline{ u}) \cdot\dd \underline{B}_{ s}}.
\end{equation*}
Considering that, under $W^{\underline{u}}_{N}$, $\dd \underline{B}_{ t}= \dd \underline{x}_{ t} -  g(\underline{x}_{ t}, \underline{u})\dd t$:
\begin{align*}
\frac{\dd P_{N}^{\underline{u}}}{\dd W^{\underline{u}}_{N}}_{\vert \mathcal{ F}_{ t}} &= \exp\left(-\int_{0}^{t}{\nabla H_{N}(\underline{x}_{ s}, \underline{ u})\cdot\dd \underline{B}_{ s}} - \frac{1}{2}\int_{0}^{t}{ \left\Vert \nabla H_{N}(\underline{x}_{ s}, \underline{ \omega}) \right\Vert^{2}\dd s}\right),\\
&=\exp\left(-\int_{0}^{t}{\nabla H_{N}(\underline{x}_{ s}, \underline{ u})\cdot \dd \underline{x}_{ s}} + \int_{0}^{t}{\nabla H_{N}(\underline{x}_{ s})\cdot g(\underline{x}_{ s}, \underline{u})\dd s} - \frac{1}{2}\int_{0}^{t}{ \left\Vert \nabla H_{N}(\underline{x}_{ s}, \underline{u}) \right\Vert^{2}\dd s}\right).
\end{align*}
Applying Ito's formula to the function $H_{N}$, we get:
\[\int_{0}^{t}{\nabla H_{N}(\underline{x}_{ s}, \underline{u})\cdot\dd \underline{x}_{ s}} = H_{N}( \underline{ x}_{t}, \underline{u}) - H_{N}( \underline{ x}_{0}, \underline{u}) - \frac{1}{2}\sum_{i=1}^{N}{\int_{0}^{t}\partial_{x_{i}}^2 H_{N}(\underline{x}_{ s}, \underline{u})\dd s},\]
so that
\begin{align}
\label{eq:dRN4}
\frac{\dd P_{N}^{\underline{u}}}{\dd W^{\underline{u}}_{N}}_{\vert \mathcal{ F}_{ t}} &= \displaystyle\exp\left(H_{N}(\underline{x}_{ 0}, \underline{u}) -
H_{N}(\underline{x}_{ t}, \underline{u}) + \int_{0}^{t}{\nabla H_{N}(\underline{x}_{ s}, \underline{u})\cdot g(\underline{x}_{ s}, \underline{u})\dd s} \right.\nonumber\\
&+\frac{1}{2} \sum_{i=1}^{N}{\int_{0}^{t}{\partial_{x_{i}}^{2}H_{N}(\underline{x}_{ s}, \underline{u})\dd s}}-\left.\frac{1}{2}\int_{0}^{t}{ \left\Vert \nabla H_{N}(\underline{x}_{ s}, \underline{u}) \right\Vert^{2}\dd s}\right).
\end{align}
A straightforward calculation yields, using \eqref{eq:symmetry_f}:
\begin{align*}
\partial_{x_i}H_{N}(\underline{ x}, \underline{ u}) &= \frac{1}{N} \sum_{ j=1}^{ N}\partial_{ x}f(x_i-x_j, u_{ i}, u_j),\ i=1, \ldots, N\\
\partial_{x_i}^2 H_{N}(\underline{ x}, \underline{ u}) &= \frac{1}{N}\sum_{j=1}^{N}{\partial_{ x}^{ 2}f(x_i-x_j, u_i, u_j)}-\frac1N \partial_{ x}^{ 2}f(0, u_i, u_i),\ i=1, \ldots, N.
\end{align*}
Proposition~\ref{prop:G1} is then simply a reformulation of \eqref{eq:dRN4} in terms of the empirical measure $L_{ N}^{ \underline{ \omega}}$.
\end{proof}
\subsection{Large deviations for truncated disorder}
\label{sec:suppcomp}
For any measure $ \nu\in \mathcal{ M}( \mathbb{ R})$ such that $ \int_{ \mathbb{ R}} \left\vert \omega \right\vert^{ k_{ 1}} \nu(\dd \omega)<+\infty$, define the following functional on $ \mathcal{ Y}$ (recall the definitions of $k_{ 1}$ in \eqref{eq:bound_g_x} and of $ \mathcal{ I}_{ \nu}$ in \eqref{eq:rate_I_no_inter}):
\begin{equation}
\label{eq:mcG_nu}
\mathcal{ G}_{ \nu}(\lambda) = \begin{cases}
\mathcal{ I}_{ \nu}( \lambda) - \mathcal{ J}(\lambda), & \text{ if } \lambda_{ 2}= \nu,\\
+\infty& \text{ otherwise}. 
\end{cases}
\end{equation}
Fix $M>0$. Let $ \chi_{ M}$ be the following truncation function
\begin{equation}
\label{eq:chi_M}
\chi_{ M}(\omega):= ( \omega \wedge M) \vee (-M),\ \omega\in \mathbb{ R}.
\end{equation}
Let $ \left\lbrace \omega_{ i}\right\rbrace_{ i\geq1}$ be the sequence of disorder introduced in Section~\ref{sec:model_assumptions}, and define the truncated sequence
\begin{equation}
\label{eq:disorder_M}
(\omega_{ M})_{ i}:= \chi_{ M}(\omega_{ i}),\ i\geq1.
\end{equation}
For simplicity, we denote by $\mathbf{P}_{N}:= \mathbf{ P}_{ N}^{ \underline{ \omega}}$ and by $\mathbf{ P}_{ N}^{ M}:= \mathbf{ P}_{ N}^{ \underline{ \omega}_{ M}}$  the law of the empirical measure $L_{N}$ of the coupled diffusions \eqref{eq:edsN} in the environments $\underline{ \omega}$ and $ \underline{ \omega}_{ M}$, respectively. Applying Proposition~\ref{prop:G1} for the sequence $ \underline{ u}= \underline{ \omega}_{ M}$, one obtains from \eqref{eq:G1}, from the definition of $ \mathcal{ K}$ and $ \mathcal{ J}$ in \eqref{eq:mcK} and \eqref{eq:mcJ} and from \eqref{eq:disorder_M} that
\begin{equation}
\label{eq:girsanov_PWM}
\frac{\dd P_{N}^{\underline{ \omega}_{ M}}}{\dd W^{\underline{ \omega}_{ M}}_{N}} = \exp\left(N \mathcal{J}_{ M}(L_{N}^{ \underline{ \omega}}) - \mathcal{K}_{ M}(L_N^{ \underline{ \omega}})\right)= \exp\left(N \mathcal{J}_{ M}(L_{N}^{ \underline{ \omega}_{ M}}) - \mathcal{K}_{ M}(L_N^{ \underline{ \omega}_{ M}})\right),
\end{equation}
where $ \mathcal{ K}_{ M}$, $ \mathcal{ J}_{ M}$ and $ \mathcal{ J}^{ (i)}_{ M}$, $i=1, \ldots, 4$ are defined by the same expressions as \eqref{eq:mcK} and \eqref{eq:mcJ} where $f(x, \omega, \tilde \omega)$ has been replaced by 
\begin{equation}
\label{eq:fM}
f_{ M}(x, \omega, \tilde \omega):=f(x, \chi_{ M}(\omega), \chi_{ M}(\tilde \omega))
\end{equation} and $g(x, \omega)$ by 
\begin{equation}
g_{ M}(x, \omega):=g(x, \chi_{ M}(\omega)).
\end{equation} Hence, from this and the assumptions on $f$ and $g$, since every function integrated in the expression of $ \mathcal{ J}_{ M}(\cdot)$ is bounded and continuous, $\mathcal{J}_{ M}: \mathcal{ Y} \to \mathbb{R}$ is bounded and continuous on $ \mathcal{ Y}$.
\begin{proposition}
\label{prop:LPDPN}
For every $M>0$, for every fixed sequence of disorder $ \left\lbrace \omega_{ i}\right\rbrace_{ i\geq1}$ satisfying the hypothesis of Section~\ref{sec:model_assumptions}, $ \left\lbrace \mathbf{P}_{N}^{ M}\right\rbrace_{ N\geq1}$ satisfies a principle of large deviation in $ \mathcal{ Y}$, with speed $N$, for the good rate function 
\begin{equation}
\label{eq:mcG_M}
\mathcal{G}_{ M}:=\mathcal{ G}_{ \mu_{ M}},
\end{equation}
given by \eqref{eq:mcG_nu} in the case $ \nu:= \mu_{ M}$, where $ \mu_{ M}$ is the law of the random variable $\chi_{ M}(\omega)$ under $ \mu$.
\end{proposition}
\begin{remark}
By definition, for any $ \lambda$ such that $ \mathcal{ G}_{ M}(\lambda)<+\infty$, $ \lambda_{ 2}= \mu_{ M}$ and $\mathcal{G}_{ M}(\lambda)=\mathcal{ G}_{ \mu_{ M}}(\lambda)= \mathcal{ I}_{ \mu_{ M}}(\lambda) - \mathcal{ J}(\lambda)$. Note first that
\begin{equation}
\mathcal{ G}_{ M}(\lambda)=\mathcal{ I}_{ \mu_{ M}}(\lambda) - \mathcal{ J}_{ M}(\lambda).
\end{equation}
Indeed, for such a $ \lambda$, $ \mathcal{ J}(\lambda)= \mathcal{ J}_{ M}(\lambda)$, as it can be verified for each term $ \mathcal{ J}^{ (i)}$, $i=1, \ldots, 4$ in \eqref{eq:mcJ}. We verify it for instance in the case of $ \mathcal{ J}^{ (2)}(\lambda)$ and leave the other terms to the reader
\begin{align*}
\mathcal{ J}^{ (2)}(\lambda)
&=\int_{0}^{T}\int_{ (\mathcal{ E}\times \mathcal{ F})^{ 2}} \partial_{ x}f(x_{s}-\tilde x_{s}, \omega, \tilde\omega)c(x_s, \omega)\lambda^{ \omega}(\dd x) \lambda^{ \tilde{ \omega}}(\dd \tilde x) \mu_{ M}(\dd\omega)\mu_{ M}(\dd\tilde\omega)\dd s,\\
&=\int_{0}^{T}\int_{ (\mathcal{ E}\times \mathcal{ F})^{ 2}} \partial_{ x}f_{ M}(x_{s}-\tilde x_{s}, \omega, \tilde \omega)c_{ M}(x_s, \omega)\lambda^{ \omega_{ M}}(\dd x) \lambda^{ \tilde{ \omega}_{ M}}(\dd \tilde x) \mu(\dd\omega)\mu(\dd\tilde\omega)\dd s,\\
&=\int_{0}^{T}\int_{ (\mathcal{ E}\times \mathcal{ F})^{ 2}} \partial_{ x}f_{ M}(x_{s}-\tilde x_{s}, \omega, \tilde \omega)c_{ M}(x_s, \omega)\lambda^{ \omega}(\dd x) \lambda^{ \tilde{ \omega}}(\dd \tilde x) \mu_{ M}(\dd\omega)\mu_{ M}(\dd\tilde\omega)\dd s,\\
&= \mathcal{ J}^{ (2)}_{M}(\lambda).
\end{align*}
\end{remark}
\begin{proof}[Proof of Proposition~\ref{prop:LPDPN}]
Fix $M>0$. Since the environment $ \left\lbrace \omega_{ i}\right\rbrace_{ i\geq1}$ satisfies the hypothesis of Section~\ref{sec:model_assumptions} for the probability measure $ \mu$, the truncated environment $ \left\lbrace \omega_{ M}\right\rbrace$ satisfies the same property for the probability measure $ \mu_{ M}$. Applying Proposition~\ref{prop:LDP_uncoupled} to $u_{ i}= \omega_{ i, M}$ and $ \nu= \mu_{ M}$, one obtains that  $  \left\lbrace \mathbf{ W}_{ N}^{ \underline{ \omega}_{ M}}\right\rbrace_{ N\geq1}$ satisfies a large deviation principle with speed $N$ with good rate function given by $ \lambda \mapsto \mathcal{ I}_{ \mu_{ M}}( \lambda)$. Applying Proposition~\ref{prop:G1} to the environment $ \underline{ \omega}_{ M}$, for all measurable subset $A$ of $ \mathcal{ Y}$,
\begin{align}
\mathbf{P}_{N}^{ M}(A)&=\int_{ \mathcal{ E}^{ N}}\mathbf{ 1}_{ \left\lbrace \underline{x},\ L_{N}^{ \underline{ \omega}_{ M}}[\underline{x}]\in A\right\rbrace}P_{N}^{\underline{\omega}_{ M}}(\dd\underline{x}),\nonumber\\
&= \int_{ \mathcal{ E}^{ N}} \mathbf{ 1}_{ \left\lbrace\underline{x},\ L_{N}^{ \underline{ \omega}_{ M}}[\underline{x}]\in A\right\rbrace} e^{N\mathcal{J}_{ M}(L_{N}^{ \underline{ \omega}_{ M}}[\underline{x}])- \mathcal{K}_{ M}(L_N^{ \underline{ \omega}_{ M}}[\underline{x}])}W^{\underline{\omega}_{ M}}_{N}(\dd\underline{x}),\label{eq:exptruc}\\
&= \int_{ A} e^{N\mathcal{J}_{ M}(\lambda)-\mathcal{K}_{ M}(\lambda)}\mathbf{W}_{N}^{ \underline{ \omega}_{ M}}(\dd\lambda).\label{eq:expG}
\end{align}
So that, by Proposition~\ref{prop:LDP_uncoupled},
\begin{align*}
-\inf{ \left\lbrace\mathcal{I}_{ \mu_{ M}}(\lambda),\ \lambda\in \mathring{A}\right\rbrace} &\leq  \liminf_{N\to\infty} \frac{1}{N}
\ln\mathbf{W}_{N}^{ \underline{ \omega}_{ M}}(A) \leq \limsup_{N\to\infty} \frac{1}{N} \ln\mathbf{W}_{N}^{ \underline{ \omega}_{ M}}(A)\\ &\leq-\inf{ \left\lbrace\mathcal{I}_{ \mu_{ M}}(\lambda),\ \lambda\in \bar{A}\right\rbrace}.
\end{align*}
The assumption made on $f$ ensures that the linear functional $\mathcal{K}_{ M}(\cdot)$ is bounded by $\frac{ \left\Vert \partial_{ x}^{ 2}f_{ M} \right\Vert_{ \infty}}{2}$ on $ \mathcal{ Y}$. Consequently, for all measurable set $A$,
\begin{equation*}
 e^{- \frac{ \left\Vert \partial_{ x}^{ 2}f_{ M} \right\Vert_{ \infty}}{ 2}}\int_A e^{N\mathcal{J}_{ M}( \lambda)}\mathbf{W}_N^{ \underline{ \omega}_{ M}}( \dd \lambda)\leq \mathbf{P}_N^{ M}(A)\leq e^{\frac{ \left\Vert \partial_{ x}^{ 2}f_{ M} \right\Vert_{ \infty} }{ 2}}\int_A e^{N\mathcal{J}_{ M}( \lambda)}\mathbf{W}_N^{ \underline{ \omega}_{ M}}(\dd \lambda),
\end{equation*}
so that
\begin{equation}
\begin{split}
\liminf_{N\to\infty} \frac{1}{N}\ln\left(\int_A e^{N\mathcal{J}_{ M}( \lambda)}\mathbf{W}_N^{ \underline{ \omega}_{ M}}(\dd \lambda)\right)&\leq \liminf_{N\to\infty} \frac{1}{N}\ln\mathbf{P}_N^{M}(A)\leq \limsup_{N\to\infty} \frac{1}{N} \ln\mathbf{P}_N^{ M}(A)\\&\leq \limsup_{N\to\infty} \frac{1}{N} \ln\left(\int_A e^{N\mathcal{J}_{ M}(\lambda)}\mathbf{W}_N^{ \underline{ \omega}_{ M}}(\dd \lambda)\right).
\end{split}
\end{equation}
Applying Varadhan's Lemma (\cite{Dembo1998}, Th. 4.3.1), the sequence $ \left\lbrace \mathbf{ P}_{ N}^{ M}\right\rbrace_{ N\geq1}$ satisfies a large deviation principle in $ \mathcal{ Y}$ for the good rate function \[\lambda\mapsto \mathcal{I}_{ \mu_{ M}}(\lambda) - \mathcal{J}_{ M}(\lambda) - \inf_{ \mathcal{ Y}}\left( \mathcal{ I}_{ \mu_{ M}}(\cdot) -\mathcal{J}_{ M}(\cdot) \right).\] 
It remains to show that $\inf_{ \mathcal{ Y}}(\mathcal{I}_{ \mu_{ M}}(\cdot)- \mathcal{J}_{ M}(\cdot))=0$: since $\mathbf{P}_{N}^{ M}$ is a
probability,
\begin{align*}
0 &= \frac{1}{N} \ln\left(\int_{ \mathcal{ Y}}e^{N\mathcal{J}_{ M}( \lambda)-\mathcal{K}_{ M}(\lambda)}\mathbf{W}_{N}^{ \underline{ \omega}_{ M}}(\dd \lambda)\right)\leq \frac{ \left\Vert \partial_{ x}^{ 2}f_{ M} \right\Vert_{ \infty}}{2N} + \frac{1}{N}\ln\left(\int_{ \mathcal{ Y}}e^{N\mathcal{J}_{ M}(\lambda)}\mathbf{W}_{N}^{ \underline{ \omega}_{ M}}(\dd \lambda)\right).
\end{align*}
Taking the limit as $N\to\infty$, one obtains: \[0\leq \sup_{ \mathcal{ Y}}\left(\mathcal{J}_{ M}(\cdot) - \mathcal{I}_{ \mu_{ M}}(\cdot)\right).\]Bounding $\mathcal{K}_{ M}(\cdot)$ by below, we have in the same way
\begin{align*}
0 &= \frac{1}{N} \ln\left(\int_{ \mathcal{ Y}}e^{N\mathcal{J}_{ M}( \lambda)-\mathcal{K}_{ M}( \lambda)}\mathbf{W}_{N}^{ \underline{ \omega}_{ M}}(\dd\lambda)\right)\geq -\frac{ \left\Vert \partial_{ x}^{ 2}f_{ M} \right\Vert_{ \infty}}{2N} + \frac{1}{N}\ln\left(\int_{ \mathcal{ Y}}e^{N\mathcal{J}_{ M}( \lambda)}\mathbf{W}_{N}^{ \underline{ \omega}_{ M}}(\dd \lambda)\right).
\end{align*}
Consequently, $\sup_{ \mathcal{ Y}} \left(\mathcal{J}_{ M}(\cdot) - \mathcal{I}_{ \mu_{ M}}(\cdot)\right) \leq 0$, so that the rate function is actually equal to
\begin{equation}
\lambda\mapsto \mathcal{I}_{ \mu_{ M}}(\lambda) - \mathcal{J}_{ M}(\lambda).
\end{equation} 
Since $  \left\vert \mathcal{ J}_{ M}(\lambda) \right\vert < +\infty$ for all $ \lambda\in \mathcal{ Y}$, $ \mathcal{ I}_{ \mu_{ M}}(\lambda) - \mathcal{ J}_{ M}(\lambda) <+ \infty \Leftrightarrow \mathcal{ I}_{ \mu_{ M}}(\lambda)< +\infty$, which is equivalent to $ \mathcal{ G}_{ M}(\lambda)< +\infty$. For such $ \lambda$, $ \mathcal{ G}_{ M}(\lambda)= \mathcal{I}_{ \mu_{ M}}(\lambda) - \mathcal{J}_{ M}(\lambda)$. Proposition~\ref{prop:LPDPN} is proven. 
\end{proof}
\subsection{Exponentially good approximations}
\label{sec:exp_coupling}
The purpose of this part is to derive a large deviation principle for $ \left\lbrace \mathbf{P}_{N}\right\rbrace_{N\geq1}$ from the sequence of large deviation principles satisfied by $ \left\lbrace \mathbf{P}_{N}^{M}\right\rbrace_{N\geq1}$ for each $M\geq1$: we introduce here an \emph{exponentially good approximations of measures}, (see \cite[\S~4.2]{Dembo1998}).  The space $ \mathcal{ Y}$ endowed with the topology of weak convergence is a metrizable space, for the following metric $d$ (see \cite[Th.~12, p.~262]{Dudley1966}):
\begin{equation}
\label{eq:distance dudley}
\forall \nu, \tilde\nu\in \mathcal{ Y},\quad d(\nu, \tilde\nu):= \sup \left\lbrace\left\vert \int{\varphi\dd\nu}-\int{\varphi\dd\tilde\nu}\right\vert,\ \varphi \in
BL_{1}( \mathcal{ E}\times \mathcal{ F})\right\rbrace,
\end{equation}
where $BL_{1}( \mathcal{ E}\times \mathcal{ F})$ is the set of all bounded Lipschitz continuous functions of Lipschitz norm bounded by $1$. Following \cite[p.~131]{Dembo1998}, we show that $ \left\lbrace \mathbf{P}_{N}^{M}\right\rbrace_{ N, M\geq1}$ are exponentially good approximations of $ \left\lbrace \mathbf{P}_{N}\right\rbrace_{ N\geq1}$:
\begin{proposition}[(Exponentially good approximations)]
\label{prop:limgam}
For all $\delta>0$, we define:
\begin{equation}
\label{eq:Gamma_delta}
\Gamma_{\delta}= \left\lbrace (\nu, \tilde{\nu})\in \mathcal{ Y}\times \mathcal{ Y},\ d(\nu, \tilde{\nu})>\delta\right\rbrace.
\end{equation}
For all $N, M\geq1$, there exists a coupling $\mathbf{Q}_{N,M}$, probability on $ \mathcal{ Y}\times \mathcal{ Y}$, such that the marginals of $\mathbf{Q}_{N,M}$ are $\mathbf{P}_{N}$ and $\mathbf{P}_{N}^{M}$ and which verifies:
\begin{equation}
\label{eq:lim_QNM_Gamma}
\forall \delta>0,\quad\lim_{M\to \infty} \limsup_{N\to\infty} \frac{1}{N} \ln \mathbf{Q}_{N,M}(\Gamma_{\delta}) = -\infty.
\end{equation}
\end{proposition}
Assuming for a moment that Proposition~\ref{prop:limgam} is true, an immediate consequence is (see \cite[Th.~4.2.16, p.~131]{Dembo1998}):
\begin{proposition}
 \label{prop:exp_approx_dembo}
The sequence $ \left\lbrace \mathbf{P}_{N}\right\rbrace_{N\geq1}$ satisfies a \emph{weak} large deviation principle in $ \mathcal{ Y}$, at speed $N$, in the weak topology, governed by the rate function:
\begin{equation}
\label{eq:G_tilde}
\mathcal{Q}(\lambda):= \sup_{\delta>0}\liminf_{M\to\infty} \inf_{ \rho\in B(\lambda,\delta)} \mathcal{ G}_{ M}(\rho).
\end{equation}
\end{proposition}
Let us now prove Proposition~\ref{prop:limgam}.
\subsubsection{Definition of the coupling $\mathbf{Q}_{N,M}$.}
\label{subsubsec:coupling}
Denote by $ \mathcal{ T}$ the measurable application
\[\begin{array}{cccl}
\mathcal{ T}:& \mathbb{ R}^{N} \times \mathcal{C}([0,T], \mathbb{R}^{N}) \times \mathbb{R}^{N} &\to& \mathcal{C}([0,T], \mathbb{ R}^N)\\
& (\underline{\xi}, \underline{B}, \underline{u}) &\mapsto & \mathcal{ T}\left[\underline{\xi}, \underline{B}, \underline{u}\right],
\end{array}\]
such that for all choice of the filtered space $(\Omega, \mathcal{F}, \mathcal{F}_{t}, \mathbf{P})$, and for all choice of a $\mathcal{F}_{t}$-adapted Brownian motion $\underline{B}$, and for all initial condition $\underline{\xi}$ which is $\mathcal{F}_{0}$ measurable, the process $ \mathcal{ T}\left[\underline{\xi}, \underline{B}, \underline{u}\right]$ is the only solution to the SDE \eqref{eq:edsN}, under the environment $\underline{u}$. Let us choose a sequence of independent standard Brownian motions $ \left\lbrace B_{ i}\right\rbrace_{ i\geq1}$ and a sequence of i.i.d uniform random variables $ \left\lbrace U_{i}\right\rbrace_{i\geq1}$ on $[0, 1]$, independent of $ \left\lbrace B_{ i}\right\rbrace_{ i\geq1}$. For each $ \omega$, denote by $t\mapsto F_{ \omega}(t)$ the cumulative distribution function of the law $ \gamma^{ \omega}$ and by $s\mapsto F_{ \omega}^{-1}(s):= \inf \left\lbrace t,\ F_{ \omega}(t)\geq s\right\rbrace$ its pseudo-inverse function. For the sequence of disorder $ \left\lbrace \omega_{ i}\right\rbrace_{ i\geq1}$ introduced in Section~\ref{sec:model_assumptions}, consider the following:
\[\left(\underline{x}, \underline{\tilde x}\right):= \left(\mathcal{ T}\left[\underline{\xi}, \underline{B}, \underline{ \omega}\right], \mathcal{ T}\left[\underline{\xi}_{ M}, \underline{B}, \underline{ \omega}_{ M}\right]\right),\]
where, for all $i\geq1$
\begin{align*}
\xi_{ i}&:= F_{ \omega_{ i}}^{ -1}(U_{ i}),\\
\xi_{ i, M}&:= F_{ \omega_{ i,M}}^{ -1}(U_{ i}).
\end{align*}
We denote by $Q_{N, M}$ the law of the processes $(\underline{x}, \underline{\tilde x})\in \left(\mathcal{ E}^{ N} \right)^{ 2}$ and by $\mathbf{Q}_{N,M}$, probability on $ \mathcal{ Y}\times \mathcal{ Y}$ the law of the corresponding couple of empirical measures $(L_{N}^{ \underline{ \omega}}[\underline{x}], L_{N}^{ \underline{ \omega}_{ M}}[\underline{\tilde x}])$. By construction, the marginals of $\mathbf{Q}_{N,M}$ are $\mathbf{P}_{N}$ and $\mathbf{P}_{N}^{M}$. In order to prove Proposition~\ref{prop:limgam}, we need the following lemma:
\begin{lemma}
\label{lem:xinf}
There exists a constant $C>0$ such that $Q_{ N, M}$ almost surely, $(x, \tilde x)$ satisfy the following property: for all $N\geq1$, all $t>0$,
\[\frac{1}{N}\sum_{ i=1}^{ N}\sup_{ s\leq t}\left\vert x_{i, t}- \tilde{x}_{i, t} \right\vert \leq e^{ Ct} \left(\frac{Ct}{N}\sum_{ i=1}^{ N}\left\vert \omega_{i}- \chi_{ M}(\omega_{ i})\right\vert \left(1+ \left\vert \omega_{ i} \right\vert^{ k_{ 2}} + \left\vert \chi_{ M}(\omega_{ i}) \right\vert^{ k_{ 2}}\right) + \frac{ 1}{ N} \sum_{ i=1}^{ N} \left\vert \xi_{ i} - \xi_{ i, M} \right\vert \right).\]
\end{lemma}
\begin{proof}[Proof of Lemma~\ref{lem:xinf}]
Let $(\underline{x}, \underline{\tilde{x}})$ sampled according to the coupling $Q_{N, M}$. Then, for all $i\in\{1,\dots, N\}$, for all $s>0$, \begin{align*}
\left\vert x_{i, s} - \tilde{x}_{i, s}\right\vert  &\leq \left\vert \xi_{ i} - \xi_{ i, M} \right\vert + \frac{1}{N}\sum_{j=1}^{N}{\int_{0}^{s}\left \vert \partial_{ x}f(x_{j,u}-x_{i,u}, \omega_i, \omega_j)- \partial_{ x}f(\tilde{x}_{j,u}-\tilde{x}_{i,u}, \chi_{ M}(\omega_i), \chi_{ M}(\omega_j))\right \vert} \dd u\\ 
&+ \int_{ 0}^{s}\left \vert g(x_{i, u}, \omega_{i})- g(\tilde x_{i, u}, \chi_{ M}(\omega_{i}))\right\vert\dd u,\\
&\leq \left\vert \xi_{ i} - \xi_{ i, M} \right\vert +  \frac{C_{ f}}{N}\sum_{ j=1}^{ N}\int_{0}^{s}  \left(\left\vert x_{i,u}-\tilde{x}_{i,u}\right\vert + \left\vert x_{j,u}-\tilde{x}_{j,u}\right\vert \right) \dd u\\ 
&+ \frac{sC_{ f}}{N}\sum_{ j=1}^{ N} \left(\left\vert \omega_{ i} - \chi_{ M}(\omega_{ i})\right\vert + \left\vert \omega_{ j}- \chi_{ M}( \omega_{ j}) \right\vert\right)\\
 &+ C_{ g}\int_{ 0}^{s} \left( \left\vert x_{ i, u} - \tilde x_{ i, u} \right\vert+\left\vert\omega_{i}- \chi_{ M}(\omega_{ i})\right\vert \left(1 + \left\vert \omega_{ i} \right\vert^{ k_{ 2}} + \left\vert \chi_{ M}(\omega_{ i}) \right\vert^{ k_{ 2}}\right)\right) \dd u.
\end{align*}
So, for all $t\in[0, T]$, if $S_{t}:= \frac{1}{N}\sum_{ i=1}^{ N}\sup_{s\leq t} \left\vert x_{i,s} - \tilde{x}_{i,s}\right\vert$ and $C_{ 1}:= 2C_{ f} + C_{ g}$,
\begin{align*}
S_{t} &\leq \frac{ 1}{ N} \sum_{ i=1}^{ N} \left\vert \xi_{ i} - \xi_{ i, M} \right\vert + \frac{t C_{ 1}}{N}\sum_{ i=1}^{ N}\left\vert \omega_{i}- \chi_{ M}(\omega_{ i})\right\vert \left(1 + \left\vert  \omega_{ i} \right\vert^{ k_{ 2}} + \left\vert \chi_{ M}(\omega_{ i}) \right\vert^{ k_{ 2}}\right)  + C_{ 1} \int_{0}^{t}{S_{ u}\dd u},\\
\end{align*}
The result follows from Gronwall's Lemma.
\end{proof}
We are now in position to prove Proposition~\ref{prop:limgam}:
\begin{proof}[Proof of Proposition~\ref{prop:limgam}]
For all $\delta>0$, 
\[\mathbf{Q}_{N,M}(\Gamma_{\delta}) = Q_{N,M}((\underline{x}, \underline{\tilde x}),\ d\left(L_{N}^{ \underline{ \omega}}[\underline{x}], L_{N}^{ \underline{ \omega}_{ M}}[\underline{\tilde x}]\right)>\delta),\] 
Since for all $\underline{x}$, $\underline{y}$ in $\mathcal{C}([0,T], \mathbb{ R}^N)$ and $\underline{u}$, $\underline{v}$ in $\mathbb{R}^N$,
\[d\left(L_{N}^{ \underline{u}}[\underline{x}], L_{N}^{ \underline{v}}[\underline{y}]\right)\leq \frac{1}{N}\sum_{i=1}^{N}{\left( \left\Vert x_{i}-y_{i} \right\Vert_{\infty}+ \left\vert u_{i}-v_{i}\right\vert \right)},\]
one has by Lemma~\ref{lem:xinf}, 
\begin{align*}
\mathbf{Q}_{N,M}(\Gamma_{\delta}) &\leq Q_{N, M}\left(\left\lbrace(\underline{x}, \underline{\tilde x}),\ \frac{1}{N}\sum_{ i=1}^{ N} \sup_{ s\leq T}\left\vert x_{i, s}-\tilde{x}_{i, s} \right\vert >\frac{\delta}{2}\right\rbrace\right) + \mathbf{ 1}_{ \left\lbrace \frac{ 1}{ N}\sum_{ i=1}^{ N} \left\vert \omega_{ i} - \chi_{ M}(\omega_{ i}) \right\vert > \frac{ \delta}{ 2}\right\rbrace},\\
&\leq Q_{ N, M} \left( \frac{ 1}{ N} \sum_{ i=1}^{ N} \left\vert \xi_{ i} - \xi_{ i, M} \right\vert > \frac{ \delta e^{ -CT}}{ 4}\right)\\
&+  \mathbf{ 1}_{ \left\lbrace \frac{1}{ N}\sum_{ i=1}^{ N} \left\vert \omega_{ i} - \chi_{ M}(\omega_{ i}) \right\vert\left(1 + \left\vert  \omega_{ i} \right\vert^{ k_{ 2}} + \left\vert \chi_{ M}(\omega_{ i}) \right\vert^{ k_{ 2}}\right)  > \frac{ \delta e^{ -CT}}{ 4CT} \right\rbrace}\\
&+  \mathbf{ 1}_{ \left\lbrace \frac{ 1}{ N}\sum_{ i=1}^{ N} \left\vert \omega_{ i} - \chi_{ M}(\omega_{ i}) \right\vert > \frac{ \delta}{ 2} \right\rbrace}.
\end{align*}
Hence, 
\begin{align}
\limsup_{ N\to \infty} \frac{ 1}{ N} \ln \mathbf{ Q}_{ N, M}( \Gamma_{ \delta}) & \leq \max \Bigg( \limsup_{ N\to\infty} \frac{ 1}{ N} \ln Q_{ N, M}\left( \frac{ 1}{ N} \sum_{ i=1}^{ N} \left\vert \xi_{ i} - \xi_{ i, M} \right\vert > \frac{ \delta e^{ -CT}}{ 4}\right),\nonumber\\
& \qquad \qquad\limsup_{ N\to\infty} \frac{ 1}{ N}\ln \mathbf{ 1}_{ \left\lbrace \frac{ 1}{ N}\sum_{ i=1}^{ N} \left\vert \omega_{ i} - \chi_{ M}(\omega_{ i}) \right\vert \left(1 + \left\vert  \omega_{ i} \right\vert^{ k_{ 2}} + \left\vert \chi_{ M}(\omega_{ i}) \right\vert^{ k_{ 2}}\right) > \frac{ \delta e^{ -CT}}{ 4CT}\right\rbrace}, \nonumber\\
& \qquad \qquad\limsup_{ N\to\infty} \frac{ 1}{ N}\ln \mathbf{ 1}_{ \left\lbrace \frac{ 1}{ N}\sum_{ i=1}^{ N} \left\vert \omega_{ i} - \chi_{ M}(\omega_{ i}) \right\vert > \frac{ \delta}{ 2}\right\rbrace}\Bigg). \label{eq:QNM_Gamma}
\end{align}
We adopt here the convention that $\ln(0)=-\infty$. We treat the three terms of \eqref{eq:QNM_Gamma} separately: fix $ \delta^{ \prime}>0$ and $r>0$. If we denote by $ \mathbf{ P}_{U}$ the law of the random variables $ \left\lbrace U_{ i}\right\rbrace_{ i\geq1}$, we have, by Markov inequality,
\begin{align*}
\frac{ 1}{ N} \ln Q_{ N, M}\left( \frac{ 1}{ N} \sum_{ i=1}^{ N} \left\vert \xi_{ i} - \xi_{ i, M} \right\vert > \delta^{ \prime}\right)&= \frac{ 1}{ N} \ln \mathbf{ P}_{U}\left( \frac{ 1}{ N} \sum_{ i=1}^{ N} \left\vert F_{ \omega_{ i}}^{ -1}(U_{ i}) - F_{ \omega_{ i, M}}^{ -1}(U_{ i}) \right\vert > \delta^{ \prime}\right),\\
&\leq \frac{ 1}{ N} \sum_{ i=1}^{ N}\ln \mathbf{ E}_{U}\left(e^{r\vert F_{ \omega_{ i}}^{-1}(U) - F_{ \omega_{ i, M}}^{-1}(U)\vert }\right) - r\delta^{ \prime}.
\end{align*}
Then for any $ \omega$,
\begin{align*}
\ln\mathbf{ E}_{U}\left(e^{r\vert F_{ \omega}^{-1}(U) - F_{ \omega_{ M}}^{-1}(U)\vert } \right) &=\ln\mathbf{ E}_{U}\left(e^{r\vert F_{ \omega}^{-1}(U) - F_{ \omega_{ M}}^{-1}(U)\vert \mathbf{ 1}_{\vert \omega\vert >M}}\right),\\
&= \mathbf{ 1}_{\vert \omega\vert >M} \ln\mathbf{ E}_{U}\left(e^{r\vert F_{ \omega}^{-1}(U) - F_{ \omega_{ M}}^{-1}(U)\vert}\right),\\
&\leq \mathbf{ 1}_{\vert \omega\vert >M} \ln\mathbf{ E}_{U}\left(e^{r\vert F_{ \omega}^{-1}(U)\vert} e^{r\vert F_{ \omega_{ M}}^{-1}(U)\vert}\right),\\
&\leq \frac{\mathbf{ 1}_{\vert \omega\vert >M}}{2} \left(\ln\mathbf{ E}_{U}\left(e^{2r\vert F_{ \omega}^{-1}(U)\vert}\right) + \ln\mathbf{ E}_{U}\left(e^{2r\vert F_{ \omega_{ M}}^{-1}(U)\vert}\right)\right),\\
&=\frac{ \mathbf{ 1}_{\vert \omega\vert >M}}{ 2} \ell_{ \omega}(2r) + \frac{ \mathbf{ 1}_{\vert \omega\vert >M}}{ 2} \ell_{ \omega_{ M}}(2r),
\end{align*}
where we recall the definition of $ \ell_{ \omega}$ in \eqref{eq:ell_exp_moments}.
Define $ \phi_{ M}, \rho_{ M}: \mathbb{ R}\to [0,1]$ by $ \phi_{M}(\omega)= \mathbf{ 1}_{ \omega>M} + (\omega-M+1) \mathbf{ 1}_{ \omega\in [M-1, M]}$, and $ \rho_{ M}( \omega) = \phi_{M}(\omega) + \phi_{M}(- \omega)$ so that
\begin{align*}
0\leq \mathbf{ 1}_{ \left\vert \omega \right\vert>M} &\leq \rho_{ M}(\omega) \leq \mathbf{ 1}_{ \left\vert \omega \right\vert>M-1} \leq 1.
\end{align*}
Consequently,
\begin{align}
\ln\mathbf{ E}_{U}\left(e^{t\vert F_{ \omega}^{-1}(U) - F_{ \omega_{ M}}^{-1}(U)\vert } \right) &\leq \frac{ \rho_{ M}(\omega)}{ 2} \ell_{ \omega}(2r) + \frac{ \rho_{M}(\omega)}{ 2}\ell_{ \omega_{ M}}(2r).\label{aux:E01}
\end{align}
Let us treat the two terms in \eqref{aux:E01} separately: by H\"older inequality, for $p>1$ defined by \eqref{eq:conv_ell} and $q$ such that $ \frac{ 1}{ p} + \frac{ 1}{ q}=1$,
\begin{align*}
\frac{ 1}{ N}\sum_{ i=1}^{ N} \frac{ \rho_{ M}(\omega_{ i})}{ 2} \ell_{ \omega_{ i}}(2r) \leq \frac{ 1}{ 2} \left( \frac{ 1}{ N} \sum_{ i=1}^{ N} \rho_{ M}(\omega_{ i})^{ q}\right)^{ \frac{ 1}{ q}} \left( \frac{ 1}{ N} \sum_{ i=1}^{ N}  \ell_{ \omega_{ i}}(2r)^{ p}\right)^{ \frac{ 1}{ p}}.
\end{align*}
By continuity of $ \rho_{ M}$ and using the assumptions \eqref{eq:convergence_mes_emp_mu} and \eqref{eq:conv_ell} on $ \left\lbrace \omega_{ i}\right\rbrace_{ i\geq1}$, one obtains that, 
\begin{align*}
\limsup_{ N\to\infty} \frac{ 1}{ N}\sum_{ i=1}^{ N} \frac{ \rho_{ M}(\omega_{ i})}{ 2}\ell_{ \omega_{ i}}(2r)\leq \frac{ 1}{ 2} \left( \int_{ \mathbb{ R}} \rho_{ M}(\omega)^{ q} \mu( \dd\omega)\right)^{ \frac{ 1}{ q}} \left( \int_{ \mathbb{ R}} \ell_{ \omega}(2r)^{ p} \mu(\dd \omega)\right)^{ \frac{ 1}{ p}},\\
\leq \frac{ 1}{ 2} \mu\left(\left\vert \omega \right\vert> M-1\right)^{ \frac{ 1}{ q}} \left( \int_{ \mathbb{ R}} \ell_{ \omega}(2r)^{ p} \mu(\dd \omega)\right)^{ \frac{ 1}{ p}}.
\end{align*}
By \eqref{eq:conv_ell}, this last quantity goes to $0$ as $M\to\infty$. We now turn to the second term in \eqref{aux:E01}: by \eqref{eq:convergence_mes_emp_mu},
\begin{align*}
\frac{ 1}{ 2} \limsup_{ N\to\infty}\frac{ 1}{ N} \sum_{ i=1}^{ N}\ell_{ \chi_{ M}( \omega_{ i})}(2r) \phi_{M}(\omega_{ i}) &= \frac{ 1}{ 2} \int_{ \mathbb{ R}} \ell_{ \chi_{ M}(\omega)}(2r) \phi_{M}(\omega) \mu(\dd \omega),\\
& \leq \frac{ 1}{ 2} \int_{ \left\vert \omega \right\vert> M-1} \ell_{ \chi_{ M}(\omega)}(2r) \mu(\dd \omega),\\
&\leq \frac{ C(2r) \left(1+  M^{ \tau}\right)}{ 2} \mu( \left\vert \omega \right\vert>M-1),\\ &\leq C(2r) \left(1+ M^{ \tau}\right)  \frac{ \mathbb{ E}_{ \mu}( \left\vert \omega \right\vert^{ \iota})}{ 2(M-1)^{ \iota}},
\end{align*}
where we used \eqref{eq:control_ell}, \eqref{eq:moment_cond_mu} and Markov inequality. Since by assumption, $ \iota> \tau$, this quantity goes to $0$ as $M\to\infty$, for fixed $r$. Consequently,
\begin{equation*}
\lim_{ M\to\infty}\limsup_{ N\to\infty} \frac{ 1}{ N} \ln Q_{ N, M}\left( \frac{ 1}{ N} \sum_{ i=1}^{ N} \left\vert \xi_{ i} - \xi_{ i, M} \right\vert > \delta^{ \prime}\right) \leq -r \delta^{ \prime}.
\end{equation*}
This is true for all $r>0$, so that
\begin{equation}
\label{eq:limsup_PU}
\lim_{ M\to\infty}\limsup_{ N\to\infty} \frac{ 1}{ N} \ln Q_{ N, M}\left( \frac{ 1}{ N} \sum_{ i=1}^{ N} \left\vert \xi_{ i} - \xi_{ i, M} \right\vert > \delta^{ \prime}\right) = -\infty.
\end{equation}
We now turn to the second (deterministic) term in \eqref{eq:QNM_Gamma}: for all $ \delta^{ \prime}>0$, $r>0$, $N\geq1$, the following inequality is true: if $s_{ N}( \underline{ \omega}):= \frac{ 1}{ N}\sum_{ i=1}^{ N} \left\vert \omega_{ i} - \chi_{ M}(\omega_{ i}) \right\vert \left(1 + \left\vert  \omega_{ i} \right\vert^{ k_{ 2}} + \left\vert \chi_{ M}(\omega_{ i}) \right\vert^{ k_{ 2}}\right)$,
\begin{align*}
\frac{ 1}{ N}\ln \mathbf{ 1}_{ \left\lbrace s_{ N}( \underline{ \omega}) > \delta^{ \prime}\right\rbrace}&\leq \frac{ 1}{ N}\ln \mathbf{ 1}_{ \left\lbrace s_{ N}( \underline{ \omega}) \geq \delta^{ \prime}\right\rbrace} \leq- r\mathbf{ 1}_{  \left\lbrace s_{ N}( \underline{ \omega}) < \delta^{ \prime}\right\rbrace}.
\end{align*}
By the lower semi-continuity of $u \mapsto \mathbf{ 1}_{ u< \delta}$ and by \eqref{eq:moment_cond_mu},
\begin{align*}
\limsup_{ N\to\infty}\frac{ 1}{ N}\ln \mathbf{ 1}_{ \left\lbrace s_{ N}( \underline{ \omega}) > \delta^{ \prime}\right\rbrace} &\leq - r \liminf_{ N\to\infty} \mathbf{ 1}_{ \left\lbrace s_{ N}( \underline{ \omega}) < \delta^{ \prime}\right\rbrace},\\
&\leq -r \mathbf{ 1}_{ \left\lbrace \int_{ \mathbb{ R}} \left\vert \omega - \chi_{ M}(\omega)\right\vert \left(1+ \left\vert \omega \right\vert^{ k_{ 2}} + \left\vert \chi_{ M}(\omega) \right\vert^{ k_{ 2}}\right)\mu(\dd \omega)< \delta^{ \prime}\right\rbrace}.
\end{align*}
Using again \eqref{eq:moment_cond_mu}, one concludes that
\begin{align*}
\lim_{ M\to\infty}\limsup_{ N\to\infty}\frac{ 1}{ N}\ln \mathbf{ 1}_{ \left\lbrace s_{ N}( \underline{ \omega})> \delta^{ \prime}\right\rbrace} &\leq -r,
\end{align*}
and, since $r>0$ is arbitrary,
\begin{equation}
\label{eq:limsup_ind_om}
\lim_{ M\to\infty}\limsup_{ N\to\infty}\frac{ 1}{ N}\ln \mathbf{ 1}_{ \left\lbrace s_{ N}( \underline{ \omega}) > \delta^{ \prime}\right\rbrace}= -\infty.
\end{equation}
From \eqref{eq:QNM_Gamma}, \eqref{eq:limsup_PU} and \eqref{eq:limsup_ind_om}, one concludes that \eqref{eq:lim_QNM_Gamma} is true. We leave to the reader the third term in \eqref{eq:QNM_Gamma} that can be estimated in a similar way. This concludes the proof of Proposition~\ref{prop:limgam}.
\end{proof}
\section{ Identification of the rate function}
\label{sec:indentification_rate_function}
The purpose of this section is to prove that the large deviation functional $\mathcal{Q}(\cdot)$ defined in \eqref{eq:G_tilde} coincides with 
\begin{equation}
\label{eq:G_I_minus_J}
\mathcal{G}(\cdot):= \mathcal{ G}_{ \mu}(\cdot)
\end{equation}
given by \eqref{eq:mcG_nu} in the case $ \nu:= \mu$.
\subsection{A simpler expression of the functionals $ \mathcal{ G}_{ \nu}$}
The main ideas of this paragraph can already be found in \cite{daiPra96}. For any $ q\in \mathcal{ M}( \mathbb{ R})$, $ \nu\in \mathcal{ M}( \mathbb{ R})$, $ \lambda\in \mathcal{ Y}$, $ \omega\in \mathbb{ R}$, let us recall the definitions of $ \beta^{ q, \omega}$ in \eqref{eq:beta_q}, of $P^{ \lambda, \omega}$ in \eqref{eq:simp1} and of $P_{ \nu}^{ \lambda}$ in \eqref{eq:def_PQ}. The following proposition provides a simpler expression of the rate function $\mathcal{G}_{ \nu}(\cdot)$ defined in \eqref{eq:mcG_nu}:
\begin{proposition} \label{prop:simrate}
For all measure $ \nu\in \mathcal{ M}( \mathbb{ R})$ such that $ \int_{ \mathbb{ R}} \left\vert \omega \right\vert^{ k_{ 1}} \nu(\dd \omega) <+\infty$, for all $\lambda\in \mathcal{ Y}$, we have
\begin{equation}
\label{eq:Inew1}
\mathcal{G}_{ \nu}(\lambda) = \mathcal{ R}_{ \nu}(\lambda)
\end{equation}
where
\begin{equation}
\label{eq:R_nu}
\mathcal{ R}_{ \nu}(\lambda)=\begin{cases}
\mathcal{H}(\lambda\vert  P^{\lambda}_{ \nu})& \text{ if } \lambda_{ 2}= \nu,\\
+\infty & \text{ otherwise}.
\end{cases}
\end{equation}
\end{proposition}
\begin{proof}[Proof of Proposition~\ref{prop:simrate}]
We follow here the ideas of \cite{daiPra96}, Lemma~2. Recall \eqref{eq:diff_copy}: for all $\omega\in\Supp(\mu)$, $W^{\omega}$ is the law of the solution to $\dd x_t= g(x_{ t}, \omega)\dd t + \dd b_{t}$. For all $ \lambda\in \mathcal{ Y}$, if $ \mathcal{ G}_{ \nu}(\lambda)< +\infty$ then $ \lambda_{ 2}= \nu$. Thanks to the moment condition on $ \nu$ and the assumptions on $f$ and $g$ made in Section~\ref{sec:model_assumptions}, this implies that $ \left\vert \mathcal{ J}(\lambda) \right\vert< +\infty$. Thus, $ \mathcal{ I}_{ \nu}(\lambda) < +\infty$ and for $ \nu$-a.e. $ \omega$, $ \lambda^{ \omega}(\dd x) \ll W^{ \omega}(\dd x)$. Since, by Girsanov Theorem, for all $ \omega$, $ W^{ \omega}(\dd x) \sim P^{m, \omega}(\dd x)$ for any $m\in \mathcal{ Y}$, this implies that for $ \nu$-a.e. $ \omega$, $ \lambda^{ \omega}(\dd x)\ll P^{\lambda, \omega}(\dd x)$, that is $ \mathcal{ R}_{ \nu}(\lambda)<+\infty$. Conversely, if $\mathcal{ R}_{ \nu}(\lambda)<+\infty$, $ \lambda_{ 2}= \nu$ (so that $ \left\vert \mathcal{ J}(\lambda) \right\vert < +\infty$) and for $ \nu$-a.e. $ \omega$, $ \lambda^{ \omega}(\dd x)\ll P^{\lambda, \omega}(\dd x)\sim W^{ \omega}(\dd x)$ so that $ \mathcal{ I}_{ \nu}(\lambda)< +\infty$. Hence, $ \mathcal{ G}_{ \nu}(\lambda)<+\infty \Leftrightarrow \mathcal{ R}_{ \nu}(\lambda)<+\infty$ and we can now restrict ourselves to $ \lambda\in \mathcal{ Y}$ such that 
\begin{equation}
\label{eq:lambda_muM}
\lambda_{ 2}= \nu \text{ and for $ \nu$-a.e. }\omega,\ \lambda^{ \omega}(\dd x)\ll W^{ \omega}(\dd x).
\end{equation}
Fix such a $ \lambda\in \mathcal{ Y}$. We first prove that
\begin{equation}
\label{eq:Fnew1}
\mathcal{J}(\lambda) = \int\ln\left(\frac{\dd P^{\lambda, \omega}}{\dd W^{\omega}}(x)\right)\lambda(\dd x, \dd\omega),
\end{equation}
where $ \mathcal{ J}$ is defined in \eqref{eq:girsanov_PWM}. By Girsanov's theorem, we have, for every $ \omega$
\begin{equation}
\label{eq:lnJ}
\ln\left(\frac{\dd P^{\lambda, \omega}}{\dd W^{\omega}}\right) = -\frac{1}{2} \int_{0}^{t}{\dd s \left(\beta^{\lambda_s, \omega} (x_{s})\right)^{2}} +
\int_{0}^{t}{\beta^{\lambda_s, \omega} (x_{s})\dd b_{s}}.
\end{equation}
The first term integrated over $\lambda$ gives rise to the term $\mathcal{J}^{ (4)}(\lambda)$ in Proposition~\ref{prop:G1}. Considering that under $W^{\omega}$, $\dd b_{t}= \dd x_{t} -g(x_t, \omega)\dd t$,
\begin{align}
\int_{0}^{t}{\beta^{\lambda_s, \omega} (x_{s})\dd b_{s}} &= \int_{0}^{t}{\beta^{\lambda_s, \omega} (x_{s})\dd x_{s}} -
\int_{0}^{t}{\beta^{\lambda_s, \omega} (x_{s})g(x_s, \omega)} \dd s. \label{aux:int_beta_g}
\end{align}
The last term of \eqref{aux:int_beta_g} integrated over $ \lambda$ gives the term $ \mathcal{ J}^{ (2)}$ in \eqref{eq:mcJ}. We now focus on the first term of the righthand side of \eqref{aux:int_beta_g}: integrating over $ \lambda$ and applying Ito's formula to the semi-martingale $(x, \tilde x)$, we have,
\begin{align*}
\int \int_{0}^{t}\beta^{\lambda_s, \omega} (x_{s})\dd x_{s} \lambda(\dd x, \dd \omega) &= \frac{1}{2} \int\int \left[f(x_{0}- \tilde x_{0}, \omega, \tilde\omega)- f(x_{t}-\tilde x_{t}, \omega, \tilde \omega)\right]\lambda(\dd x, \dd\omega)\lambda(\dd \tilde x, \dd\tilde\omega) \\ &+ \frac{1}{2} \int\int\left[\int_{0}^{t}\partial_{ x}^{ 2}f(x_{s}-\tilde x_{s}, \omega, \tilde\omega)\dd s\right]\lambda(\dd x, \dd\omega) \lambda(\dd \tilde x,\tilde\omega).
\end{align*}
This proves equality \eqref{eq:Fnew1}. We finish by proving equality \eqref{eq:Inew1} for $ \lambda$ satisfying \eqref{eq:lambda_muM}:
\[\mathcal{G}_{ \nu}(\lambda) = \int \left\lbrace\ln\left(\frac{\dd \lambda^{ \omega}}{\dd W^{\omega}}\right)(x)+ \ln\left(\frac{\dd W^{\omega}}{\dd P^{\lambda, \omega}}\right)(x)\right\rbrace \lambda(\dd x, \dd \omega)= \mathcal{ R}_{ \nu}(\lambda).\]
Proposition~\ref{prop:simrate} follows.
\end{proof}
\subsection{Identification of the rate function}
At this point, Proposition~\ref{prop:exp_approx_dembo} states a \emph{weak} large deviation principle for $ \left\lbrace \mathbf{ P}_{ N}\right\rbrace_{ N\geq1}$ with rate function
$\mathcal{Q}(\cdot)$ defined in \eqref{eq:G_tilde}. The next step is to identify $\mathcal{Q}(\cdot)$ with the rate function $\mathcal{G}$ defined in \eqref{eq:G_I_minus_J}. This is the purpose of Propositions~\ref{prop:ineqI1} and~\ref{prop:ineqI2} below:
\begin{proposition}
\label{prop:ineqI1}
For all $\lambda\in \mathcal{ Y}$,
\begin{equation}
\label{eq:identI1}
\mathcal{Q}(\lambda) \leq \mathcal{ G}(\lambda).
\end{equation}
\end{proposition}
\begin{proof}[Proof of Proposition~\ref{prop:ineqI1}]
It is preferable to use here the representation of $ \mathcal{ G}= \mathcal{ G}_{ \mu}$ given in \eqref{eq:mcG_nu}. We only need to consider the case $ \mathcal{ G}(\lambda) <+\infty$, which means that $ \lambda_{ 2}= \mu$ and by \eqref{eq:bound_g_x} and \eqref{eq:moment_cond_mu}, $\left\vert \mathcal{ J}(\lambda)  \right\vert< +\infty$. Hence, $ \mathcal{ I}_{ \mu}(\lambda) <+\infty$ so that for $ \mu$-almost every $ \omega$, $ \lambda^{ \omega}(\dd x) \ll W^{\omega}(\dd x)$. 
We define, for all $M>0$, $ \kappa_{M}\in \mathcal{ Y}$ by:
\begin{equation}
\label{eq:rhoM}
\kappa_M(\dd x, \dd\omega):= \tilde\lambda^{\omega}_{ M}(\dd x)\mu_{ M}(\dd \omega),
\end{equation}
where
\begin{equation}
\tilde \lambda^{ \omega}_{ M}(\dd x)=\begin{cases}
\lambda^{ \omega}(\dd x), & \text{ if } \left\vert \omega \right\vert < M,\\
W^{\omega}(\dd x), & \text{ if } \left\vert \omega \right\vert\geq M.
\end{cases}
\end{equation}
For such a $ \kappa_M$, $\mathcal{G}_{M}( \kappa_M)<\infty$: indeed, $ \kappa_{ M, 2}= \mu_{ M}$ so that $ \left\vert \mathcal{ J}(\kappa_{ M}) \right\vert< +\infty$ and for $ \mu_{ M}$-almost every $\omega$, $ \kappa_M^{\omega}\ll  W^{\omega}$. 

Moreover, there exists $M_{0}$ such that, for all $M\geq M_{0}$,  $ \kappa_{M}\in B(\lambda,\delta)$: indeed, for $\varphi\in BL_{1}(\mathcal{ E}\times \mathcal{ F})$, we have successively:
\begin{align*}
A_M(\varphi) &:= \left\vert \int{\varphi(x, \omega) \kappa_M(\dd x, \dd\omega)} - \int{\varphi(x,\omega)\lambda(\dd x, \dd\omega)}\right\vert ,\\
&= \left\vert \int_{ \mathcal{ F}} \left\lbrace\int_{ \mathcal{ E}}\varphi(x, \chi_{ M}(\omega))\tilde\lambda^{ \chi_{ M}(\omega)}_{ M}(\dd x) - \int_{ \mathcal{ E}} \varphi(x, \omega) \lambda^{ \omega}(\dd x) \right\rbrace\mu(\dd\omega) \right\vert ,\\
&= \left\vert \int_{ (-M, M)^{ c}} \left\lbrace\int_{ \mathcal{ E}}\varphi(x, \chi_{ M}(\omega))\tilde\lambda^{ \chi_{ M}(\omega)}_{ M}(\dd x) - \int_{ \mathcal{ E}} \varphi(x, \omega) \lambda^{ \omega}(\dd x) \right\rbrace\mu(\dd\omega) \right\vert,\\
&\leq 2 \left\Vert \varphi \right\Vert_{ \infty} \mu((-M, M)^{ c})\leq 2\mu((-M, M)^{ c}).
\end{align*}
This last quantity is independent of $ \varphi\in BL_{ 1}( \mathcal{ E}\times \mathcal{ F})$ and goes to $0$ as $M\to\infty$. So $ \kappa_{M}\in B(\lambda,\delta)$ for $M$ sufficiently large.

Secondly, we have the following equality:\begin{equation}
\label{eq:liminfR}
\liminf_{M\to\infty} \mathcal{ G}_{ M}(\kappa_{ M}) =\mathcal{ G}(\lambda).
\end{equation}
We first prove that
\begin{equation}
\lim_{ M\to\infty} \mathcal{ I}_{ \mu_{ M}}( \kappa_{ M}) = \mathcal{ I}_{ \mu}(\lambda).
\end{equation}
Indeed,
\begin{align*}
\mathcal{ I}_{ \mu_{ M}}(\kappa_{ M})&=\int_{\mathbb{R}} \mathcal{H}(\tilde\lambda^{ \omega}_{ M}\vert  W^{\omega})\mu_M(\dd\omega)=  \int_{ \mathbb{ R}} \mathcal{ H}(\tilde\lambda^{ \chi_{ M}(\omega)}_{ M} \vert W^{\chi_{ M}(\omega)}) \mu(\dd \omega).
\end{align*}
For fixed $ \omega$, $\mathcal{ H}(\tilde\lambda^{ \chi_{ M}(\omega)}_{ M} \vert W^{\chi_{ M}(\omega)})$ tends to $\mathcal{ H}(\lambda^{\omega} \vert W^{\omega})$, as $M\to\infty$ (the two terms are actually equal for large $M$): more precisely,
\begin{align*}
\mathcal{ H}( \tilde \lambda^{ \chi_{ M}(\omega)}_{ M} \vert W^{\chi_{ M}(\omega)})&= \begin{cases}
\mathcal{ H}(\lambda^{ \omega} \vert W^{\omega}) & \text{ if } \left\vert \omega \right\vert< M,\\
\mathcal{ H}( W^{\chi_{ M}(\omega)} \vert W^{\chi_{ M}(\omega)})=0 & \text{ if } \left\vert \omega \right\vert\geq M
\end{cases}
\end{align*}
In any case, $\mathcal{ H}( \tilde \lambda^{ \chi_{ M}(\omega)}_{ M} \vert W^{\chi_{ M}(\omega)})\leq \mathcal{ H}(\lambda^{ \omega} \vert W^{\omega})$. Since by hypothesis $ \mathcal{ I}_{ \mu}(\lambda)= \int_{ \mathbb{ R}} \mathcal{ H}( \lambda^{ \omega}\vert W^{\omega}) \mu(\dd\omega)<+\infty$, an application of the dominated convergence theorem leads to
\begin{equation} 
\mathcal{ I}_{ \mu_{ M}}(\kappa_{ M}) = \int_{ \mathbb{ R}} \mathcal{ H}( \tilde \lambda^{ \omega}_{ M} \vert W^{\omega}) \mu_{ M}(\dd \omega) \to \int_{ \mathbb{ R}} \mathcal{ H}(\lambda^{ \omega} \vert W^{\omega}) \mu(\dd \omega)= \mathcal{ I}_{ \mu}(\lambda),
\end{equation}
as $M\to\infty$. 
It remains to prove that $ \mathcal{ J}_{ M}(\kappa_{ M})$ converges to $ \mathcal{ J}(\lambda)$ as $M\to\infty$. It suffices to prove it for each $ \mathcal{ J}^{ (i)}_{M}(\kappa_{ M})$, for $i=1, \ldots, 4$ (recall \eqref{eq:mcJ}). We only prove it for $i=2$ and leave the other terms to the reader:
\begin{align*}
\left\vert \mathcal{ J}^{ (2)}_{M}( \kappa_{ M}) - \mathcal{ J}^{ (2)}(\lambda)\right\vert &\leq  \int_{0}^{T}\int_{ \mathbb{ R}^{ 2}}\Bigg\vert \int_{\mathcal{ E}^{ 2}} \partial_{ x}f(x_{s}-\tilde x_{s}, \omega_{ M}, \tilde\omega_{ M})g(x_s, \omega_{ M}) \tilde{ \lambda}_{ M}^{ \omega_{ M}}(\dd x) \tilde{ \lambda}_{ M}^{ \omega_{ M}}(\dd \tilde{ x}) \\& -\int_{\mathcal{ E}^{ 2}} \partial_{ x}f(x_{s}-\tilde x_{s}, \omega, \tilde\omega)g(x_s, \omega)\lambda^{ \omega}(\dd x)\lambda^{ \omega}(\dd \tilde{ x}) \Bigg\vert \mu(\dd \omega) \mu(\dd \tilde{ \omega}) \dd s.
\end{align*}
Let us decompose the integration domain of the disorder $\mathbb{ R}^{ 2}$ into $O_{ 1}\sqcup O_{ 2}\sqcup O_{ 3}\sqcup O_{ 4}$ where
\begin{align*}
O_{ 1}&:=(-M, M)^{ 2}\\
O_{ 2}&:=(-M, M)\times (-M, M)^{ c}\\
O_{ 3}&:=(-M, M)^{ c}\times (-M, M)\\
O_{ 4}&:=\left\lbrace(-M, M)^{ c}\right\rbrace^{ 2}.
\end{align*} 
By definition of $ \kappa_{ M}$, the contribution of $O_{ 1}$ is zero and the integrasl over $O_{ i}$, $i=2, 3, 4$ go to zero as $M\to\infty$. We only treat the case of $O_{ 3}$ and leave the rest to the reader:
\begin{align*}
d_{ M, 3}&:=\int_{0}^{T}\int_{ (-M, M)^{ c} \times (-M, M)}\Bigg\vert \int_{\mathcal{ E}^{ 2}} \partial_{ x}f(x_{s}-\tilde x_{s}, \omega_{ M}, \tilde\omega_{ M})g(x_s, \omega_{ M}) \tilde{ \lambda}_{ M}^{ \omega_{ M}}(\dd x) \tilde{ \lambda}_{ M}^{ \omega_{ M}}(\dd \tilde{ x}) \\& -\int_{\mathcal{ E}^{ 2}} \partial_{ x}f(x_{s}-\tilde x_{s}, \omega, \tilde\omega)g(x_s, \omega)\lambda^{ \omega}(\dd x)\lambda^{ \omega}(\dd \tilde{ x}) \Bigg\vert \mu(\dd \omega) \mu(\dd \tilde{ \omega}) \dd s,\\
&\leq C_{ f} C_{ g}T\int_{(-M, M)^{ c}\times (-M, M)} \left(M^{ k_{ 1}} + \left\vert \omega \right\vert^{ k_{ 1}}\right) \mu(\dd \omega) \mu(\dd \tilde{ \omega}),\\
&= C_{ f} C_{ g}T \left(M^{ k_{ 1}} \mu \left( \left\vert \omega \right\vert \geq M\right) \mu( \left\vert \omega \right\vert < M) + \mu( \left\vert \omega \right\vert < M) \int_{ \left\vert \omega \right\vert\geq M} \left\vert \omega \right\vert^{ k_{ 1}} \mu(\dd \omega)\right)
\end{align*}
which, by Markov inequality, using \eqref{eq:moment_cond_mu} goes to $0$ as $M\to\infty$, since $ \iota> k_{ 1}$. Hence, \eqref{eq:liminfR} is true. We are now in position to prove \eqref{eq:identI1}: fix $ \delta>0$, and $M_{ 0}$ such that $ \kappa_{ M}\in B(\lambda, \delta)$ for all $M\geq M_{ 0}$. We have for all $M\geq M_{ 0}$,
\begin{equation*}
\inf_{ \kappa\in B(\lambda, \delta)} \mathcal{ G}_{ M}(\kappa) \leq \mathcal{ G}_{ M}( \kappa_{ M}),
\end{equation*}
So,
\begin{equation*}
\liminf_{M\to\infty} \inf_{ \kappa\in B(\lambda, \delta)} \mathcal{ G}_{ M}(\kappa) \leq \liminf_{M\to\infty} \mathcal{ G}_{ M}(\kappa_{ M}) \leq \mathcal{ G}(\lambda).
\end{equation*}
Taking the supremum on $\delta>0$ in the last inequality, we have the result.\end{proof}
\begin{proposition}
\label{prop:ineqI2}
For all $\lambda\in \mathcal{ Y}$,
\begin{equation}
\label{eq:identI2}
\mathcal{Q}(\lambda) \geq \mathcal{ G}(\lambda).
\end{equation}
\end{proposition}
\begin{proof}[Proof of Proposition~\ref{prop:ineqI2}]
We use here the representation \eqref{eq:R_nu}. Fix $ \omega\in \Supp(\mu)$. We claim that the functional $ \lambda \mapsto \mathcal{ H}( \lambda^{ \omega}\ \vert P^{ \lambda, \omega})$ is lower semicontinuous. It suffices to prove that it is sequentially lower semicontinuous: choose $ \left\lbrace \lambda_{ p}\right\rbrace_{ p\geq1}$ such that $ \lambda_{ p}$ converges weakly to $ \lambda$ as $ p\to \infty$. The following usual representation for the entropy holds (see for example \cite{Dembo1998}, Lemma~6.2.13):
\begin{equation}
\mathcal{ H}( \lambda^{ \omega}\vert P^{ \lambda, \omega}) = \sup_{ \phi} \left\lbrace \int_{ \mathcal{ E}} \phi(x) \lambda^{ \omega}(\dd x) - \ln \int_{ \mathcal{ E}} e^{ \phi(x)}P^{ \lambda, \omega}(\dd x)\right\rbrace,
\end{equation}
where the supremum holds on the set of continuous and bounded test functions $ \phi$. Then, 
\begin{align*}
\liminf_{ p\to\infty} \mathcal{ H}( \lambda_{ p}^{ \omega} \vert P^{ \lambda_{ p}, \omega})&= \sup_{ q\geq0} \inf_{ p\geq q} \sup_{ \phi} \left\lbrace \int_{ \mathcal{ E}} \phi(x) \lambda^{ \omega}_{ p}(\dd x) - \ln \int_{ \mathcal{ E}} e^{ \phi(x)} P^{ \lambda_{ p}, \omega}(\dd x)\right\rbrace,\\
&\geq \sup_{ \phi} \sup_{ q\geq0} \inf_{ p\geq q}\left\lbrace \int_{ \mathcal{ E}} \phi(x) \lambda^{ \omega}_{ p}(\dd x) - \ln \int_{ \mathcal{ E}} e^{ \phi(x)} P^{ \lambda_{ p}, \omega}(\dd x)\right\rbrace,\\
&= \sup_{ \phi} \liminf_{ p\to\infty} \left\lbrace \int_{ \mathcal{ E}} \phi(x) \lambda^{ \omega}_{ p}(\dd x) - \ln \int_{ \mathcal{ E}} e^{ \phi(x)} P^{ \lambda_{ p}, \omega}(\dd x)\right\rbrace,\\
&=\sup_{ \phi} \left\lbrace \int_{ \mathcal{ E}} \phi(x) \lambda^{ \omega}(\dd x) - \ln \int_{ \mathcal{ E}} e^{ \phi(x)} P^{ \lambda, \omega}(\dd x)\right\rbrace= \mathcal{ H}( \lambda^{ \omega}\vert P^{ \lambda, \omega}),
\end{align*}
where we used in the last equality that $P^{ \lambda_{ p}, \omega}$ converges weakly to $P^{ \lambda, \omega}$ if $ \lambda_{ p}\to \lambda$. This proves the claim.

Let us fix now $ \lambda\in \mathcal{ Y}$ such that $\mathcal{ Q}(\lambda)< +\infty$ and $ \varepsilon>0$. By definition of $\mathcal{ Q}(\lambda)$ in \eqref{eq:G_tilde}, there exists a sequence $ \left\lbrace k_{ M}\right\rbrace_{ M\geq1}$ such that $k_{ M}\to\infty$ as $M\to\infty$ and $ \left\lbrace \pi_{ M}\right\rbrace_{ M\geq1}\in \mathcal{ Y}^{ \mathbb{ N}}$ such that $ \pi_{ M}$ converges weakly to $ \lambda$ as $M\to\infty$, such that
\begin{align*}
\mathcal{ G}_{ k_{ M}}( \pi_{ M}) = \int_{ \mathbb{ R}} \mathcal{ H} \left( \pi_{ M}^{ \chi_{ k_{ M}}(\omega)} \vert P^{ \pi_{ M}, \chi_{ k_{ M}}(\omega)}\right)\mu(\dd \omega)\leq \mathcal{ Q}(\lambda) + \varepsilon.
\end{align*}
For fixed $ \omega$,
\begin{align}
\liminf_{ M\to\infty} \mathcal{ H} \left(\pi_{ M}^{ \chi_{ k_{ M}}(\omega)} \vert P^{ \pi_{ M}, \chi_{ k_{ M}}(\omega)}\right) &= \liminf_{ M\to\infty} \mathcal{ H} \left(\pi_{ M}^{\omega} \vert P^{ \pi_{ M}, \omega}\right), \label{aux:H_rhoM_1}\\
&\geq \mathcal{ H} \left( \lambda^{ \omega} \vert P^{ \lambda, \omega}\right) \label{aux:H_rhoM_2},
\end{align}
where we used that $\mathcal{ H} \left(\pi_{ M}^{ \chi_{ k_{ M}}(\omega)} \vert P^{ \pi_{ M}, \chi_{ k_{ M}}(\omega)}\right)$ is actually equal to $\mathcal{ H} \left(\pi_{ M}^{\omega} \vert P^{ \pi_{ M}, \omega}\right)$ for a sufficiently large $M$ in \eqref{aux:H_rhoM_1} and the lower semicontinuity of $ \lambda \mapsto \mathcal{ H} \left( \lambda^{ \omega} \vert P^{ \lambda, \omega}\right)$ in \eqref{aux:H_rhoM_2}. By \eqref{aux:H_rhoM_2} and Fatou's Lemma,
\begin{align*}
\mathcal{ G}(\lambda)&=\int_{ \mathbb{ R}} \mathcal{ H}( \lambda^{ \omega} \vert P^{ \lambda, \omega}) \mu(\dd \omega),\\
&\leq \int_{ \mathbb{ R}} \liminf_{ M\to\infty} \mathcal{ H} \left(\pi_{ M}^{ \chi_{ k_{ M}}(\omega)} \vert P^{ \pi_{ M}, \chi_{ k_{ M}}(\omega)}\right)\mu(\dd \omega),\\
&\leq \liminf_{ M\to\infty} \int_{ \mathbb{ R}} \mathcal{ H} \left(\pi_{ M}^{ \chi_{ k_{ M}}(\omega)} \vert P^{ \pi_{ M}, \chi_{ k_{ M}}(\omega)}\right)\mu(\dd \omega)= \liminf_{ M\to\infty} \mathcal{ G}_{ k_{ M}}( \pi_{ M}),\\
&\leq\mathcal{ Q}(\lambda) + \varepsilon.
\end{align*}
Since this is true for all $ \varepsilon>0$, we obtain that $ \mathcal{ G}(\lambda) \leq \mathcal{ Q}(\lambda) <+\infty$. Proposition~\ref{prop:ineqI2} follows.
\end{proof}
\begin{proposition}
\label{prop:G_coercive}
The functional $ \mathcal{ G}$ is a good rate function and the weak LDP of Proposition~\ref{prop:exp_approx_dembo} is a strong large deviation principle.
\end{proposition}
\begin{proof}[Proof of Proposition~\ref{prop:G_coercive}]
This property is quite general and relies on the fact that the space $ \mathcal{ Y}$ is a Polish space. We know that $ \left\lbrace \mathbf{ P}_{ N}^{ M}\right\rbrace_{ N\geq1}$ satisfies for all $M$ a large deviation principle with good rate function $ \mathcal{ G}_{ M}$ and that $ \left\lbrace \mathbf{ P}_{ N}^{ M}\right\rbrace_{ N\geq1}$ is an exponentially good approximation of $ \left\lbrace\mathbf{ P}_{ N}\right\rbrace_{ N\geq1}$ in the sense of Proposition~\ref{prop:limgam}. One can deduce from this that the sequence $ \left\lbrace \mathbf{ P}_{ N}\right\rbrace_{ N\geq 1}$ is exponentially tight, namely, that for all $ \alpha< \infty$, there exists a compact $K_{ \alpha}$ in $ \mathcal{ Y}$ such that 
\begin{equation}
\label{eq:exp_tight}
\limsup_{ N\to\infty} \frac{ 1}{ N} \ln \mathbf{ P}_{ N}( K_{ \alpha}^{ c}) \leq - \alpha. 
\end{equation}
If we assume for the moment that \eqref{eq:exp_tight} is true, Proposition~\ref{prop:G_coercive} follows from standard arguments: applying the large deviation lower bound for $ \left\lbrace \mathbf{ P}_{ N}\right\rbrace_{ N\geq1}$ to the open set $K_{ \alpha}^{ c}$, one obtains that $ - \alpha > \liminf_{ N\to\infty} \frac{ 1}{ N} \ln \mathbf{ P}_{ N}( K_{ \alpha}^{ c}) \geq - \inf_{ \lambda \in K_{ \alpha}^{ c}} \mathcal{ G}( \lambda)$. In particular, $ \alpha< \inf_{ \lambda\in K_{ \alpha}^{ c}} \mathcal{ G}(\lambda)$, so that the $\alpha$-level set of $ \mathcal{ G}$, $  \left\lbrace \lambda\in \mathcal{ Y},\ \mathcal{ G}(\lambda)\leq \alpha\right\rbrace$ is included in the compact set $K_{ \alpha}$. Hence, $ \mathcal{ G}$ is a good rate function. The upper bound for closed sets $F$ follow from the weak upper bound for compact sets and the obvious estimate $\limsup_{ N\to\infty} \frac{ 1}{ N}\ln \mathbf{ P}_{ N}(F) \leq \max \left( \limsup_{ N\to\infty} \frac{ 1}{ N}\ln \mathbf{ P}_{ N}(F\cap K_{ \alpha}), \limsup_{ N\to\infty} \frac{ 1}{ N}\ln \mathbf{ P}_{ N}(K_{ \alpha}^{ c})\right)$.

It remains to prove \eqref{eq:exp_tight}: for any $N, M\geq1$, denote by $( \lambda_{ N}^{ M}, \lambda_{ N})\in \mathcal{ Y}^{ 2}$ any vector with law $ \mathbf{ Q}_{ N, M}$, where $ \mathbf{ Q}_{ N, M}$ is the coupling introduced in Proposition~\ref{prop:limgam}. Then, for any $k\geq1$, for any $m_{ 1}, \ldots, m_{ k}\in \mathcal{ Y}$, any $ \eta>0$, \[ \left\lbrace \lambda_{ N}^{ M} \in \bigcup_{ i=1}^{ k} B(m_{ i}, \eta)\right\rbrace \subset \left\lbrace \lambda_{ N}\in \bigcup_{ i=1}^{ k} B(m_{ i}, 2 \eta)\right\rbrace \cup \left\lbrace (\lambda_{ N}^{ M}, \lambda_{ N}) \in \Gamma_{ \eta}\right\rbrace,\] where $B(m, \eta)$ is the open ball in $ \mathcal{ Y}$ of center $m$ and radius $ \eta$ and $ \Gamma_{ \eta}$ is defined in \eqref{eq:Gamma_delta}. Consequently,
\begin{align*}
\limsup_{ N\to\infty} \frac{ 1}{ N} \ln \mathbf{ P}_{ N} \left( \left[\bigcup_{ i=1}^{ k} B(m_{ i}, 2\eta)\right]^{ c}\right)&\leq \max \Bigg( \limsup_{ N\to\infty} \frac{ 1}{ N} \ln \mathbf{ P}_{ N}^{ M} \left( \left[\bigcup_{ i=1}^{ k} B(m_{ i}, \eta)\right]^{ c }\right),\\
&\qquad\qquad \limsup_{ N\to\infty} \frac{ 1}{ N} \ln \mathbf{ Q}_{ N, M}( \Gamma_{ \eta})\Bigg).
\end{align*}
Fix $ \alpha>0$ and $ \eta>0$. By \eqref{eq:lim_QNM_Gamma}, one can choose $M\geq1$ sufficiently large so that $\limsup_{ N\to\infty} \frac{ 1}{ N} \ln \mathbf{ Q}_{ N, M}( \Gamma_{ \eta})\leq - \alpha -1$. For this $M\geq1$, the sequence $ \left\lbrace\mathbf{ P}_{ N}^{ M}\right\rbrace $ satisfies a large deviation principle with a good rate function. Hence, by \cite{Dembo1998}, Ex. 4.1.10, (b), there exists $k\geq1$ large enough and $m_{ 1}, \ldots, m_{ k}\in \mathcal{ Y}$ such that 
\[ \limsup_{ N\to\infty} \frac{ 1}{ N} \ln \mathbf{ P}_{ N}^{ M} \left( \left[ \bigcup_{ i=1}^{ k} B(m_{ i}, \eta)\right]^{ c}\right)< - \alpha.\]
From the two previous estimates, we deduce that
\[\limsup_{ N\to\infty} \frac{ 1}{ N} \ln \mathbf{ P}_{ N} \left( \left[\bigcup_{ i=1}^{ k} B(m_{ i}, 2\eta)\right]^{ c}\right) < - \alpha.\]
An application of \cite{Dembo1998}, Ex. 4.1.10, (a) shows that the sequence $ \left\lbrace\mathbf{ P}_{ N}\right\rbrace_{ N\geq1}$ is exponentially tight.
\end{proof}
\section{ Large deviations of the empirical flow}
\label{sec:proof_flow}
This section if devoted to prove Theorem~\ref{theo:tilde_LN_x_om}. As a corollary of Theorem~\ref{theo:LN_x_om}, an easy application of the contraction principle shows that for any sequence $ \left\lbrace\omega_{ i}\right\rbrace_{ i\geq1}$ satisfying the assumptions of Section~\ref{sec:model_assumptions}, the sequence $ \left\lbrace \mathbf{ p}_{ N}^{ \underline{ \omega}}\right\rbrace_{ N\geq1}$ satisfies a large deviation principle in $ \mathcal{ C}([0, T], \mathcal{ M}(\mathbb{ R}\times \mathbb{ R}))$ with good rate function
\begin{equation}
\label{eq:rate_function_flow_2}
\mathscr{ G}_{ 1}: q\in \mathcal{ C}([0, T], \mathcal{ M}(\mathbb{ R}\times \mathbb{ R})) \mapsto \inf \left\lbrace \mathcal{ G}(\lambda),\ \pi\lambda= q\right\rbrace,
\end{equation}
where $ \mathcal{ G}(\cdot)$ is defined in \eqref{eq:rate_function_proc} and $ \pi$ in \eqref{eq:mapping_mcE_to_C}.  The main issue is the identification of the rate function $ \mathscr{ G}_{ 1}$ with $ \mathscr{ G}$ defined in \eqref{eq:rate_function_mscG}. The first result follows the lines of \cite{daiPra96}:
\begin{proposition}
\label{prop:G2_G3_equal}
If $q\in \mathcal{ C}([0, T], \mathcal{ M}(\mathbb{ R}\times \mathbb{ R}))$ is such that $ \mathscr{ G}_{ 1}(q)< \infty$, then 
\begin{equation}
\label{eq:G2_equal_G3}
\mathscr{ G}_{ 1}(q)= \mathscr{ G}(q)
\end{equation}
\end{proposition}
\begin{proof}[Proof of Proposition~\ref{prop:G2_G3_equal}]
For any $q\in \mathcal{ C}([0, T], \mathbb{ R} \times \mathbb{ R})$ such that $ \mathscr{ G}_{ 1}(q)<\infty$, since $ \mathcal{ G}$ is a good rate function, there exists $ \lambda\in \mathcal{ M}( \mathcal{ C}([0, T], \mathbb{ R}), \mathbb{ R})$ such that $ \pi\lambda= q$ and $ \mathcal{ G}( \lambda) = \mathscr{ G}_{ 1}(q) < \infty$. In particular, $ \lambda_{ 2}= \mu$ and
\begin{equation}
\mathcal{ G}( \lambda) = \int_{ \mathbb{ R}} \mathcal{ H}( \lambda^{ \omega} \vert W^{ \omega}) \mu(\dd \omega) - \mathcal{ J}( \lambda)
\end{equation}
Notice that the quantity $ \mathcal{ J}(\lambda)$ (recall \eqref{eq:mcJ}) only depends on $ q$. Therefore, we use the notation 
\begin{equation}
\mathcal{ J}(\lambda) =: \mathscr{ J}(q).
\end{equation}
For $ \mu$-a.e. $ \omega$, $ \lambda^{ \omega}$ minimises $ \mathcal{ H}( \lambda^{ \omega} \vert W^{ \omega})$ under the constraint $ \pi \lambda^{ \omega} = q^{ \omega}$. This implies (\cite{MR983373}, p.~165) that $ \lambda^{ \omega}$ is the law of a diffusion 
\begin{equation}
\label{eq:diff_lambda_om}
\dd x_{ t}^{ \omega}= b^{ \omega}_{ t}(x_{ t}^{ \omega}) \dd t + d w_{ t},
\end{equation} where $ w_{ t}$ is a standard Brownian motion. Girsanov formula implies that
\begin{equation}
\frac{ \dd \lambda^{ \omega}}{ \dd W^{ \omega}}= \exp \left(\int_{ 0}^{T} \left[b^{ \omega}_{ t}(\cdot) - g(\cdot, \omega)\right] \dd \tilde{ w}_{ t} + \frac{ 1}{ 2} \int_{ 0}^{T} \left[b^{ \omega}_{ t}(\cdot) - g(\cdot, \omega)\right]^{ 2} \dd t\right) \frac{ \dd q^{ \omega}_{ 0}}{  \dd \gamma^{ \omega}},
\end{equation}
where $ \tilde{ w}$ is a standard Brownian motion under $ \lambda^{ \omega}$. Hence,
\begin{equation}
\mathcal{ H}( \lambda^{ \omega}\vert W^{ \omega})= \mathcal{ H}( q^{ \omega}_{ 0}\vert \gamma^{ \omega}) + \frac{ 1}{ 2} \int \int_{ 0}^{T} \left[b^{ \omega}_{ t}(x_{ t}) - g(x_{ t}, \omega)\right]^{ 2}\dd t \lambda^{ \omega}(\dd x),
\end{equation}
so that, if $q_{ 0}(\dd x):= \int_{ \mathbb{ R}} q_{ 0}^{ \omega}(\dd x) \mu(\dd \omega)$, using \eqref{eq:Fnew1}
\begin{align*}
\mathscr{ G}_{ 1}(q)&= \mathcal{ H}( q_{ 0}\vert \gamma) + \frac{ 1}{ 2} \int_{ \mathbb{ R}} \int_{ 0}^{T} \int_{ \mathbb{ R}}\left[b_{ t}^{ \omega}(x) - g(x, \omega)\right]^{ 2} q_{ t}^{ \omega}( \dd x)\dd t \mu(\dd \omega) - \mathscr{ J}(q),\\
&= \mathcal{ H}( q_{ 0}\vert \gamma) + \frac{ 1}{ 2} \int_{ \mathbb{ R}} \int_{ 0}^{T} \int_{ \mathbb{ R}}\left[b_{ t}^{ \omega}(x) - g(x, \omega) - \beta^{ q_{ t}, \omega}(x)\right]^{ 2} q_{ t}^{ \omega}( \dd x)\dd t \mu(\dd \omega).
\end{align*}
Recall the definition of $ \mathscr{ K}(q, \omega)$ in \eqref{eq:def_mscK}. Considering tests functions of the form $c \phi$, $c\in \mathbb{ R}$, and taking the supremum on $c$ first and then on $ \phi$ leads to the following alternative representation of $ \mathscr{ K}$:
\begin{equation}
\mathscr{ K}(q, \omega)= \frac{ 1}{ 2}\sup_{ \phi\in \mathcal{ C}_{ 0}^{ \infty}(]0, T[\times \mathbb{ R})} \frac{ \left(\int_{ 0}^{T} \left\langle \phi(t, \cdot)\, ,\, \partial_{ t}q_{ t}^{ \omega} - \mathscr{ L}^{ \omega}q_{ t}^{ \omega}\right\rangle\dd t\right)^{ 2}}{ \int_{ 0}^{T} \left\langle (\partial_{ x} \phi(t, \cdot))^{ 2}\, ,\, q_{ t}^{ \omega}\right\rangle}.
\end{equation}
The flow $q^{ \omega}= \pi \lambda^{ \omega}$, law of the diffusion \eqref{eq:diff_lambda_om}, is the weak solution of the PDE
\begin{equation}
\partial_{ t}q_{ t}^{ \omega} = \frac{ 1}{ 2} \partial_{ x}^{ 2}q_{ t}^{ \omega} - \partial_{ x} \left( b^{ \omega}_{ t}(\cdot) q_{ t}^{ \omega}\right),
\end{equation}
Consequently, an application of Cauchy-Schwartz inequality leads to
\begin{align}
\mathscr{ K}(q, \omega)&= \frac{ 1}{ 2} \sup_{ \phi \in \mathcal{ C}_{ 0}^{ \infty}(]0, T[ \times \mathbb{ R})} \frac{ \left( \int_{ 0}^{T}\left\langle \phi(t, \cdot)\, ,\, -\partial_{ x} \left(q_{ t}^{ \omega} (b_{ t}^{ \omega} - g(\cdot, \omega) - \beta^{ q_{ t}, \omega})\right)\right\rangle \dd t\right)^{ 2}}{ \int_{ 0}^{T} \left\langle (\partial_{ x}\phi(t, \cdot))^{ 2}\, ,\, q_{ t}^{ \omega}\right\rangle \dd t}, \nonumber\\
&=\frac{ 1}{ 2} \sup_{ \phi \in \mathcal{ C}_{ 0}^{ \infty}(]0, T[ \times \mathbb{ R})} \frac{ \left( \int_{ 0}^{T}\left\langle \partial_{ x}\phi(t, \cdot)\, ,\,  q_{ t}^{ \omega} (b_{ t}^{ \omega} - g(\cdot, \omega) - \beta^{ q_{ t}, \omega})\right\rangle \dd t\right)^{ 2}}{ \int_{ 0}^{T} \left\langle (\partial_{ x}\phi(t, \cdot))^{ 2}\, ,\, q_{ t}^{ \omega}\right\rangle \dd t}, \nonumber\\
&\leq \frac{ 1}{ 2} \int_{ 0}^{T}\int_{ \mathbb{ R}}\left[b_{ t}^{ \omega}(x) - g(x, \omega) - \beta^{ q_{ t}, \omega}(x)\right]^{ 2} q_{ t}^{ \omega}( \dd x) \dd t.\label{aux:CS_qt}
\end{align}
Hence, it suffices to prove that \eqref{aux:CS_qt} is actually an equality. This is indeed true as the set $\{ \partial_{ x}\phi, \phi\in \mathcal{ C}_{ 0}^{ \infty}\}$ is dense in $L^{ 2}(q_{ t}^{ \omega}(\dd x) \dd t)$ (see \cite{daiPra96},~p.749 for details).
\end{proof}
Thus, it remains to prove that 
\begin{proposition}
\label{prop:cattiaux_leonard}
Let $q\in \mathcal{ C}([0, T], \mathcal{ M}(\mathbb{ R} \times \mathbb{ R}))$. If $ \mathscr{ G}(q)<\infty$, then $ \mathscr{ G}_{ 1}(q)<\infty$.
\end{proposition}
Proposition~\ref{prop:cattiaux_leonard} is actually a difficult result, which is related to the existence of Nelson processes for diffusions, and is now well covered in the literature (see \cite{MR1262893,MR1307957} and references therein). An alternative proof for this result would be to use the  of Dawson and G\"artner \cite{MR885876} (see also \cite{Muller:2015aa}), that requires several alternative representations of the rate function $ \mathscr{ G}$ (see in particular \cite{MR885876}, \S~4.3 and 4.4). Although the techniques in \cite{MR885876} look perfectly applicable to our disordered case, the approach of \cite{MR1262893} seems to be more direct here.
\begin{proof}[Proof of Proposition~\ref{prop:cattiaux_leonard}]
Most of the ingredients of the following proof are contained in \cite{MR1262893} (see in particular the proof of Theorem~5.9). Note that the functional $ \mathscr{ G}$ in \eqref{eq:rate_function_mscG} can be written as
\begin{equation}
\mathscr{ G}(q)= \begin{cases}
\mathcal{ H}(q_{ 0}\vert \gamma) + \int_{ \mathbb{ R}} \mathscr{ K}(q, u) \mu(\dd u)& \text{ if } q\in \mathbb{ A},\\
+\infty& \text{ otherwise},
\end{cases}
\end{equation}
where $ \mathscr{ K}$ defined in \eqref{eq:def_mscK} can be alternatively given by, for $ \mu$-almost every $u$,
\begin{align}
\mathscr{ K}(q, u)&= \sup_{ \phi\in \mathcal{ C}_{ 0}^{ \infty}(]0, T[ \times \mathbb{ R})} \Big\lbrace \int \phi(x, T) q_{ T}^{ u}(\dd x) - \int \phi(x, 0) q_{ 0}^{ u}(\dd x) \nonumber\\ &- \int_{ 0}^{T} \int \left( \partial_{ t} + (\mathscr{ L}^{ u})^{ \ast}\right)\phi(x, t) q_{ t}^{ u}( \dd x)\dd t - \frac{ 1}{ 2} \int_{ 0}^{T}\int (\partial_{ x} \phi(x, t))^{ 2} q_{ t}^{ u}(\dd x)\dd t\Big\rbrace.
\end{align}
Suppose now that $ \mathscr{ G}(q)<\infty$. Then, one can write $q: t \mapsto q_{ t}^{ u}(\dd x) \mu(\dd u)$ and $ \mathscr{ K}(q, u)<\infty$ for $ \mu$-almost every $u$. In particular, there exists a constant $C_{ u}>0$ such that for all test function $ \phi$,
\begin{align*}
\mathcal{ U}(\phi)&:= \int \phi(x, T) q_{ T}^{ u}(\dd x) - \int \phi(x, 0) q_{ 0}^{ u}(\dd x) - \int_{ 0}^{T} \int \left(\partial_{ t} + (\mathscr{ L}^{ u})^{ \ast}\right)\phi(x, t) q_{ t}^{ u}(\dd x) \dd t,
\end{align*}is such that
\begin{align*}
\mathcal{ U}(\phi) \leq C_{ u} + \frac{ 1}{ 2} \left\Vert \partial_{ x}\phi \right\Vert_{ L^{ 2}_{ q^{ u}}}^{ 2}.
\end{align*}Since $ \mathcal{ U}$ is linear, this implies that $ \mathcal{ U}$ can be extended to a continuous linear form on the completion of $ \mathcal{ C}_{ 0}^{ \infty}(]0, T[\times \mathbb{ R})$ equipped with the seminorm $ \left\Vert \partial_{ x}\phi \right\Vert_{ L_{ q^{ u}}^{ 2}}$. We denote by $H_{ 0}^{ -1}(q^{ u})$ this completion. By Riesz representation theorem, there exists $B^{ u}\in H_{ 0}^{ -1}(q^{ u})$ such for all $ \phi\in \mathcal{ C}_{ 0}^{ \infty}(]0, T[ \times \mathbb{ R})$, 
\begin{multline}
\int \phi(x, T) q_{ T}^{ u}(\dd x) - \int \phi(x, 0) q_{ 0}^{ u}(\dd x) - \int_{ 0}^{T} \int \left(\partial_{ t} + \left( \mathscr{ L}^{ u}\right)^{ \ast}\right) \phi(x, t) q_{ t}^{ u}(\dd x) \dd t \\= \int_{ 0}^{T} \int \left(\partial_{ x} \phi B^{ u}\right)(x, t) q_{ t}^{ u}(\dd x) \dd t.
\end{multline}
In particular, 
\begin{align}
\mathscr{ K}(q, u) = \sup_{ \phi\in \mathcal{ C}_{ 0}^{ \infty}(]0, T[ \times \mathbb{ R})} \left\lbrace \left\langle B^{ u}\, ,\, \partial_{ x} \phi\right\rangle_{ L_{ q^{ u}}^{ 2}} - \frac{ 1}{ 2} \left\Vert \partial_{ x}\phi \right\Vert_{ L_{ q^{ u}}^{ 2}}^{ 2}\right\rbrace \geq \frac{ 1}{ 2} \left\Vert B^{ u} \right\Vert_{ L_{ q^{ u}}^{ 2}}^{ 2}. \label{aux:mscK_B}
\end{align}
Following the same procedure as in \cite{MR1262893}, Th. 5.9, 3) (ii) (see also \cite{MR1262893} (2.6)), it is possible to show that there exists $ \lambda_{ \ast}^{ u}$ such that $ \pi \lambda_{ \ast}^{ u}= q_{ u}$ and $ \mathcal{ H}( \lambda_{ \ast}^{ u}\vert P^{ q, u})= \frac{ 1}{ 2} \left\Vert B^{ u} \right\Vert_{ L_{ q^{ u}}^{ 2}}^{ 2}$. From \eqref{aux:mscK_B}, we obtain that $ \mathscr{ K}(q, u) \geq  \mathcal{ H}( \lambda_{ \ast}^{ u}\vert P^{ q, u})$ so that $ \infty > \mathscr{ G}(q) \geq  \mathcal{ G}( \lambda_{ \ast})$ for $ \lambda_{ \ast}= \lambda_{ \ast}^{ u} \mu(\dd u)$. Thus, $ \mathscr{ G}_{ 1}(q)< \infty$.
\end{proof}
\appendix
\section{ Proof of the quenched Sanov Theorem}
\label{sec:proof_Sanov}
We sketch the proof of Proposition~\ref{prop:leonard}. Recall the notations of Section~\ref{sec:abstract_sanov}. Proposition~\ref{prop:leonard} uses large deviation techniques for projective limits developed by Dawson and G\"artner in \cite{MR885876}. The crucial result is:
\begin{theorem}[(\cite{MR885876}, Th. 3.4)]
\label{theo:DawsonGaertner}
Let $ \left\lbrace m_{ N}\right\rbrace_{ N\geq1}$ a sequence of probability measures on $ \mathcal{ Y}$ and $ \left\lbrace\gamma_{ N}\right\rbrace_{ N\geq1}$ a sequence of probability measures such that $ \lim_{ N\to \infty} \gamma_{ N}= \infty$. Suppose that the following conditions are satisfied:
\begin{enumerate}
\item for each $ \phi\in \mathcal{ W}$, the limit 
\begin{equation}
\label{eq:Lambda_phi}
\Lambda( \phi):= \lim_{ N\to \infty} \frac{ 1}{ \gamma_{ N}} \ln \int_{ \mathcal{ Y}} \exp \left( \gamma_{ N} \left\langle \lambda\, ,\, \phi\right\rangle\right) m_{ N}(\dd \lambda)
\end{equation}exists and is finite,
\item the function $ \Lambda$ is G\^ateaux differentiable.

Define
\begin{equation}
\label{eq:Lambda_ast}
\Lambda^{ \ast}( \lambda):= \sup_{ \phi \in \mathcal{ W}} \left\lbrace \left\langle \lambda\, ,\, \phi\right\rangle - \Lambda(\phi)\right\rbrace,\ \lambda\in \mathcal{ X},
\end{equation}
and suppose further that
\item
$\{ \lambda\in \mathcal{ X},\ \Lambda^{ \ast}( \lambda)< +\infty\} \subset \mathcal{ Y}$,
\end{enumerate}
then $ \left\lbrace m_{ N}\right\rbrace_{ N\geq1}$ satisfies a strong large deviation principle in $ \mathcal{ Y}$ with speed $ \gamma_{ N}$ with rate function given by the restriction of $ \Lambda^{ \ast}$ to $ \mathcal{ Y}$.
\end{theorem}
Let us verify the assumptions of Theorem~\ref{theo:DawsonGaertner} in the context of Proposition~\ref{prop:leonard}: we first prove the existence of the limit $ \Lambda(\cdot)$ in \eqref{eq:Lambda_phi}. Fix $ \phi\in \mathcal{ W}=  \mathcal{ C}_{ b}( \mathcal{ E}\times \mathcal{ F})$. For all $N\geq1$, using the independence of the $\{x_{ i}\}$,  
\begin{align*}
\Lambda_{ N}( \phi)&:= \frac{ 1}{ N} \ln \int_{ \mathcal{ Y}} \exp( N \left\langle \lambda\, ,\, \phi\right\rangle) m_{ N}(\dd \lambda),\\
&= \frac{ 1}{ N} \ln \mathbf{ E}_{ \otimes_{ i=1}^{ \infty} \rho^{ u_{ i}}} \left( \exp\left(\sum_{ i=1}^{ N} \phi(x_{ i}, u_{ i})\right)\right),\\
&=\frac{ 1}{ N} \sum_{ i=1}^{ N}\ln \int_{ \mathcal{ E}}\exp\left(\phi(x, u_{ i})\right) \rho^{ u_{ i}}(\dd x).
\end{align*}
Since $ u \mapsto \ln \int_{ \mathcal{ E}} \exp( \phi(x, u)) \rho^{ u}(\dd x)$ is continuous and bounded, assumption \eqref{eq:hyp_ui} implies that
\begin{equation}
\Lambda_{ n}( \phi) \to \Lambda(\phi):= \int_{ \mathcal{ F}}  \ln \int_{ \mathcal{ E}} \exp( \phi(x, u)) \rho^{ u}(\dd x) \nu(\dd u).
\end{equation}
Note that $ \left\vert \Lambda(\phi) \right\vert \leq \left\Vert \phi \right\Vert_{ \infty}$ for all $ \phi\in \mathcal{ W}$, so $ \Lambda( \phi)$ is finite for all $ \phi\in \mathcal{ W}$.

Moreover, $ \Lambda$ is G\^ateaux differentiable: it is easy to see that for all $ \phi, \psi$, $t\in \mathbb{ R}$, $ t \mapsto \Lambda( \phi+ t \psi)$ is differentiable in $t$ and
\begin{align}
\frac{ \dd}{ \dd t} \left[ \Lambda( \phi+ t \psi)\right]= \int_{ \mathcal{ F}}  \frac{ \int_{ \mathcal{ E}} \psi(x, u) e^{ \phi(x, u) + t \psi(x, u)} \rho^{ u}(\dd x) }{ \int_{ \mathcal{ E}} e^{ \phi(x, u) + t \psi(x, u)} \rho^{ u}(\dd x)}\nu(\dd u).
\end{align}
Finally, let $ \lambda\in \mathcal{ X}$ such that $ \Lambda^{ \ast}(\lambda)<+\infty$. Considering test functions in \eqref{eq:Lambda_ast} of the form $a \phi_{ 0}$ with $a\leq 0$ and $ \phi_{ 0}\geq0$, one obtains that $ \left\langle \lambda\, ,\, \phi_{ 0}\right\rangle\geq0$ for every $ \phi_{ 0}\geq0$, which gives $ \lambda\geq0$.
Moreover, taking constant function $ \phi\equiv c$ one obtains $ \left\langle \lambda\, ,\, 1\right\rangle=1$.
Finally, for any sequence $\left\lbrace \phi_{ n}\right\rbrace$ in $ \mathcal{ W}$ such that $ \phi_{ n}\geq0$ for all $n$ and $ \phi_{ n}(x, u)$ decreases to $0$ for fixed $(x, u)\in \mathcal{ E}\times \mathcal{ F}$, it is direct to see that $\lim_{ n\to \infty} \left\langle \lambda\, ,\, \phi_{ n}\right\rangle= \limsup_{ n\to\infty} \left\langle \lambda\, ,\, \phi_{ n}\right\rangle=0$. Therefore, one can uniquely identify $ \lambda$ with an element of $ \mathcal{ Y}$ (see \cite{MR0272004}, Proposition II-7-2).
Furthermore, the following holds
\begin{equation}
\label{eq:marg_lambda}
\forall \lambda\in \mathcal{ X},\ \Lambda^{ \ast}(\lambda)<+\infty \Rightarrow \lambda_{ 2}= \nu.
\end{equation}
Indeed, taking $ \phi(x, u)= h(u)$ independent of $x$, 
\begin{align*}
\Lambda^{ \ast}(\lambda) &\geq \sup_{ h\in \mathcal{ C}_{ b}(\mathcal{ F})}\left\lbrace \left\langle \lambda\, ,\, h\right\rangle - \left\langle \nu\, ,\, h\right\rangle\right\rbrace,\\
&= \begin{cases} 0 & \text{ if } \lambda_{ 2}= \nu,\\
+\infty & \text{ otherwise},
\end{cases}
\end{align*}
which gives \eqref{eq:marg_lambda}.
\begin{lemma}
\label{lem:IH}
For any measure $ \nu\in \mathcal{ Y}$, extend the definition of the relative entropy $ m \in \mathcal{ Y}\mapsto H(m \vert \nu)$ to $ \mathcal{ X}$ by setting $H(\cdot \vert \nu)= +\infty$ outside $ \mathcal{ Y}$. Then, for all $ \lambda\in \mathcal{ X}$,
\begin{equation}
I(\lambda) = \Lambda^{ \ast}(\lambda)
\end{equation}
where $ \mathcal{ I}$ is defined in \eqref{eq:gen_rate_I}.
\end{lemma}

\begin{proof}[Proof of Lemma~\ref{lem:IH}]
 Let us prove that for all $ \lambda \in \mathcal{ X}$,
 \begin{equation}
\Lambda^{ \ast}(\lambda) \leq I(\lambda).
 \end{equation}
It suffices to consider $ \lambda\in \mathcal{ X}$ such that $I(\lambda)<+\infty$. This implies in particular that $ \lambda\in \mathcal{ Y}$ with $ \lambda_{ 2}(\dd u) = \nu(\dd u)$. Since $ \mathcal{ E} \times \mathcal{ F}$ is a Polish space, there exists a disintegration of $ \lambda$ in terms of
\begin{equation}
\label{eq:dis_lambda}
\lambda( \dd x, \dd u)= \lambda^{ u}(\dd x) \lambda_{ 2}(\dd u)
\end{equation} where $ u \mapsto \lambda^{ u}(\dd x)$ is measurable. For such a $ \lambda$,
\begin{align}
\Lambda^{ \ast}(\lambda) &= \sup_{ \phi\in \mathcal{ W}} \left\lbrace \int_{ \mathcal{ F}} \nu(\dd u) \left(\int_{ \mathcal{ E}} \phi(x, u) \lambda^{ u}(\dd x) - \ln \int_{ \mathcal{ E}} \exp( \phi(x, u)) \rho^{ u}(\dd x) \right)\right\rbrace,\nonumber\\
&\leq \int_{ \mathcal{ F}} \nu(\dd u) \sup_{ f\in \mathcal{ C}_{ b}(\mathcal{ E})} \left\lbrace \int_{ \mathcal{ E}} f(x) \lambda^{ u}(\dd x) - \ln \int_{ \mathcal{ E}} \exp(f(x)) \rho^{ u}(\dd x) \right\rbrace, \nonumber\\
&= \int_{ \mathcal{ F}} \nu(\dd u) H \left(\lambda^{ u} \vert \rho^{ u}\right),\label{aux:repr_H}\\
&= \int_{ \mathcal{ F}} \nu(\dd u) H \left(\lambda^{ u} \vert \rho^{ u}\right) + \underbrace{H(\lambda_{ 2} \vert \nu)}_{ =0},\ \text{ since } \lambda_{ 2}= \nu, \nonumber\\
&= H(\lambda\vert \rho^{ u}\times \nu)= I(\lambda). \nonumber
\end{align}
where we used \cite{Dembo1998}, Lemma~6.2.13 in \eqref{aux:repr_H}.
Let us prove the converse inequality: $I(\lambda) \leq \Lambda^{ \ast}(\lambda)$. One can again suppose that $ \Lambda^{ \ast}(\lambda)<+\infty$, so that $ \lambda\in \mathcal{ Y}$ and the disintegration \eqref{eq:dis_lambda} holds, with $ \lambda_{ 2}= \nu$, by \eqref{eq:marg_lambda}. By Jensen's inequality, for any $ \phi\in \mathcal{ W}$, \begin{equation}
\Lambda(\phi) \leq \ln \int_{ \mathcal{ E}\times \mathcal{ F}} e^{ \phi(x, u)} \rho^{ u}(\dd x) \nu(\dd u)
\end{equation}  so that, for any such $ \lambda$
\begin{align*}
\Lambda^{ \ast}(\lambda)& \geq \sup_{ \phi\in \mathcal{ W}} \left\lbrace \int_{ \mathcal{ E} \times\mathcal{ F}}\phi(x, u) \lambda^{ u}(\dd x) \nu(\dd u)- \ln \int_{ \mathcal{ E}\times\mathcal{ F}} \exp( \phi(x, u)) \rho^{ u}(\dd x)\nu(\dd u) \right\rbrace,\\
&= H( \lambda \vert \rho^{ u}\times \nu)= I(\lambda).\qedhere
\end{align*}
\end{proof}

\end{document}